\documentclass[3p,12pt,preprint,sort&compress,reqno]{elsarticle}

\usepackage{amsmath,amsthm,amsfonts,amssymb,bm}
\usepackage{mathrsfs}
\usepackage{latexsym}

\usepackage{hyperref}
\usepackage{color}

\usepackage{indentfirst} 
\usepackage[titletoc]{appendix} 

\allowdisplaybreaks[4]
\numberwithin{equation}{section}
\def\R{{\mathbb R}}

\let\eps=\varepsilon

\newcommand{\p}{\partial}
\newcommand{\f}{\frac}
\newcommand{\Z}{\Bbb Z}
\newcommand{\cF}{\mathcal{F}}

\newcommand{\Forall}{{~\forall\,}}

\newcommand{\beq}{\begin{equation}}
\newcommand{\eeq}{\end{equation}}
\newcommand{\beqo}{\begin{equation*}}
\newcommand{\eeqo}{\end{equation*}}

\newcommand{\abs}[1]{\left\vert#1\right\vert}

\newcommand{\normm}[1]{\left\Vert#1\right\Vert}
\newcommand{\inner}[1]{\left(#1\right)}

\newtheorem{thm}{Theorem}[section]  
\newtheorem{mylem}{Lemma}[section]

\newtheorem{remark}{Remark}[section]

\makeatletter
\def\ps@pprintTitle{%
     \let\@oddhead\@empty
     \let\@evenhead\@empty
     \def\@oddfoot{\reset@font\hfil\thepage\hfil}
     \let\@evenfoot\@oddfoot}

\makeatother

\begin{document}

\begin{frontmatter} 
\title{On the hydrostatic approximation of 3D Oldroyd-B model}

\author[ad1]{Marius Paicu}
\ead{marius.paicu@math.u-bordeaux.fr}
  
\author[ad2]{Tianyuan Yu}
\ead{tianyuanyu@dlmu.edu.cn}

\author[ad3]{Ning Zhu\corref{cor1}}
\ead{ning.zhu@sdu.edu.cn, mathzhu1@163.com}

\cortext[cor1]{Corresponding author}

\address[ad1]{Universit\'e de Bordeaux, Institut de Math\'ematiques de Bordeaux, F-33405 Talence Cedex, France}

\address[ad2]{School of Science, Dalian Maritime University, Dalian, 116026, People's Republic of China}

\address[ad3]{School of Mathematics, Shandong University, Jinan, Shandong, 250100, People's Republic of China}

\begin{abstract}
In this paper, we study the hydrostatic approximation for the 3D Oldroyd-B model. Firstly, we derive the hydrostatic approximate system for this model and prove the global well-posedness of the limit system with small analytic initial data in horizontal variable. Then we justify the hydrostatic limit strictly from the re-scaled Oldroyd-B model to the hydrostatic Oldroyd-B model and obtain the precise convergence rate. 
\end{abstract}

 
\end{frontmatter}

\section{Introduction}
We consider the Oldroyd-B model in a three-dimensional thin strip $\Omega^\eps=\R^2\times\eps\mathbb{T}$:
\beq\label{Oldroyd-B}
\left\{
\begin{split}
&\p_tU+U\cdot\nabla U-\nu\Delta U+\nabla P=\operatorname{div}\mathcal{T},\\
&\p_t\mathcal{T}+U\cdot\nabla\mathcal{T}+Q(\mathcal{T},\nabla U)+a\mathcal{T}=\mu_1\mathbb{D}(U),\\
&\nabla\cdot U=0,
\end{split}
\right.
\eeq
where $U(t,x,y)$ and $P(t,x,y)$ stand the velocity field and pressure of the fluid, respectively. $\mathcal{T}(t,x,y)$ stands for the non-Newtonian part of the stress tensor which is a symmetric matrix. The bilinear form $Q$ is determined by 
\[
Q(\mathcal{T},\nabla U)=b(\mathbb{D}(U)\mathcal{T}+\mathcal{T}\mathbb{D}(U))+\mathcal{T}\Omega(U)-\Omega(U)\mathcal{T}.
\]
Here the parameter $b\in[-1,1]$, $\mathbb{D}(U)$ and $\Omega(U)$ stand the symmetric part and skew-symmetric part of $\nabla U$, respectively. In other words, it holds that
\begin{align*}
\mathbb{D}(U)=\f12\left(\nabla U+(\nabla U)^{T}\right),\quad \Omega(U)=\f12\left(\nabla U-(\nabla U)^{T}\right).
\end{align*}
The parameters $\nu>0$, $\mu_1>0$ and $a>0$ are determined by 
\[
\nu=\frac{\theta}{\operatorname{\bf Re}},\quad a=\frac{1}{\operatorname{\bf We}}, \quad \mu_1=\frac{2(1-\theta)}{\operatorname{\bf We}\operatorname{\bf Re}},
\] 
where $\operatorname{\bf Re}$ and $\operatorname{\bf We}$ are the Reynolds and Weissenberg numbers of the fluid, respectively. Finally, $\theta$ is the ratio between the so-called relaxation and retardation times. 

The Oldroyd-B model was first introduced by Oldroyd in 1950 (cf.\cite{Oldroyd1950}) to describe a typical constitutive model that does not satisfy Newtonian laws. It has been widely studied in the past decades. When $b=0$, Lions--Masmoudi (cf.\cite{LionsMasmoudi2000}) established that the weak solution is global well-posedness. The case for $b\neq0$ is still open up to now. Guillop\'e--Saut (cf.\cite{GuillopeSaut1990}) and Hieber--Naito--Shibata (cf.\cite{HieberNaitoShibata2012}) showed the global well-posedness of strong solution for Oldroyd-B model in smooth bounded domain and exterior domain when the initial data and the coupling parameter $\theta$ are small sufficiently, respectively. Later, Molinet--Talhouk (cf.\cite{MolinetTalhouk2004}) and Hieber--Fang--Zi (cf.\cite{FangHieberZi2013}) removed the smallness assumption of the coupling parameter for smooth bounded domain and exterior domain, respectively. When the initial data is restricted in scaling invariant spaces, Chemin--Masmoudi (cf.\cite{CheminMasmoudi2001}) proved the local well-posedness results in critical Besov spaces (See also \cite{ChenMiao2008}). Moreover, they showed the solution is in fact global provided the initial data and the coupling parameter is small sufficiently. Later, Zi--Fang--Zhang (cf.\cite{ZiFangZhang2014}) removed the smallness assumption for the coupling parameter. In recent years, De Anna--Paicu (cf.\cite{DePaicu2020}) established the Fujita-Kato theory for some generalized Oldroyd-B model. Kessenich (cf.\cite{kessenich2009}) and Cai--Lei--Lin--Masmoudi (cf.\cite{cailei2019}) showed the vanishing viscosity limit results for 3D and 2D incompressible viscoelastic fluids, respectively. Recently, Zi (cf.\cite{Zi2021}) justified the vanishing viscosity limit when the coupling parameter $\theta\rightarrow0$ in the framework of analytic spaces. Other interesting results involving Oldroyd-B model can be found in \cite{BreschPrange2014,ConstantinKliegl2012,ElgindiLiu2015,EligndiRousset2015,Fangzi2016,FengWangWu,LeiLiuZhou2008,zhu2018,WWX,ZW} and the reference therein.


In this work, we prescribe the Oldroyd-B model \eqref{Oldroyd-B} the initial data
\[
U\big|_{t=0}=\left(u_{0}\left(x,\frac{y}{\eps}\right), \eps v_0\left(x,\frac{y}{\eps}\right)\right), \quad \mathcal{T}\big|_{t=0}=
\begin{pmatrix}
    \eps\tau_{11}^0\left(x,\dfrac{y}{\eps}\right) & \eps\tau_{12}^0\left(x,\dfrac{y}{\eps}\right) & \eps\tau_{13}^0\left(x,\dfrac{y}{\eps}\right) \\[6pt]
    \eps\tau_{21}^0\left(x,\dfrac{y}{\eps}\right) & \eps\tau_{22}^0\left(x,\dfrac{y}{\eps}\right) & \eps\tau_{23}^0\left(x,\dfrac{y}{\eps}\right) \\[6pt]
    \eps\tau_{31}^0\left(x,\dfrac{y}{\eps}\right) & \eps\tau_{32}^0\left(x,\dfrac{y}{\eps}\right) & \eps\tau_{33}^0\left(x,\dfrac{y}{\eps}\right)
\end{pmatrix},
\]
where $u_0^\varepsilon,v_0^\varepsilon$ stand the tangential and normal velocities and we also used $x,y$ to represent the horizontal and vertical variable, respectively. In this work, we impose the following assumptions:
\[
\operatorname{\bf Re}=\eps^{-2},\quad\operatorname{\bf We}=\eps,
\]
and the coupling parameter $0<\theta<1$ is a constant. Moreover, to guarantee the uniqueness of the equations and compute the pressure gradient, we assume that
\[
\int_{\eps\mathbb{T}}U_2(t,x,y)\mathrm{d}y=0.
\] 

Inspired by \cite{BL92, Lions96}, we write
\begin{align*}
U=(u^\eps,\eps v^\eps)\left(t,x,\frac{y}{\eps}\right), \quad P=p^\eps\left(t,x,\frac{y}{\eps}\right), \quad 
\end{align*}
and
\begin{align*}
\mathcal{T}=\begin{pmatrix}
 \eps\tau_{11}^\eps\left(t,x,\dfrac{y}{\eps}\right) & \eps\tau_{12}^\eps\left(t,x,\dfrac{y}{\eps}\right) & \eps\tau_{13}^\eps\left(t,x,\dfrac{y}{\eps}\right)\\[6pt]
 \eps\tau_{21}^\eps\left(t,x,\dfrac{y}{\eps}\right) & \eps\tau_{22}^\eps\left(t,x,\dfrac{y}{\eps}\right) & \eps\tau_{23}^\eps\left(t,x,\dfrac{y}{\eps}\right)\\[6pt]
 \eps\tau_{31}^\eps\left(t,x,\dfrac{y}{\eps}\right) & \eps\tau_{32}^\eps\left(t,x,\dfrac{y}{\eps}\right) & \eps\tau_{33}^\eps\left(t,x,\dfrac{y}{\eps}\right)
\end{pmatrix},
\end{align*}
where $u^\eps=(u_1^\eps,u_2^\eps)$ and $v^\eps$ stand the tangential and normal velocities, respectively. Then the velocity field $(u^\eps,v^\eps)$ satisfies the following equations in $\R^2\times\mathbb{T}$:
\beq\label{Oldroyd-B-Scaled1}
\left\{
\begin{split}
&\left(\p_t+u^\eps\cdot\nabla_x+v^\eps\p_y-\theta\Delta_{\eps}\right)u^\eps+\nabla_x p^\eps=\begin{pmatrix}
 \eps\p_{x_1}\tau_{11}^\eps+\eps\p_{x_2}\tau_{12}^\eps+\p_y\tau_{13}^\eps\\[6pt]
 \eps\p_{x_1}\tau_{21}^\eps+\eps\p_{x_2}\tau_{22}^\eps+\p_y\tau_{23}^\eps
\end{pmatrix},\\
&\eps^2\left(\p_t+u^\eps\cdot\nabla_x+v^\eps\p_y-\theta\Delta_{\eps}\right)v^\eps+\p_y p^\eps=\eps(\eps\p_{x_1}\tau_{31}^\eps+\eps\p_{x_2}\tau_{32}^\eps+\p_y\tau_{33}^\eps),\\
&\nabla_x\cdot u^\eps+\p_yv^\eps=0, \quad \int_{\mathbb{T}}v^\eps(t,x,y)\mathrm{d}y=0,
\end{split}
\right.
\eeq
where $\Delta_{\eps}=\eps^2\Delta_x+\p_y^2$ with $\Delta_x=\p_{x_1}^2+\p_{x_2}^2$. If $\tau=0$ in the above equation, then it reduces to the classical hydrostatic Navier-Stokes equations which have been widely applied to depict the flow of atmosphere and ocean, where the vertical scale is quite small compared to the horizontal one. Similar to the velocity field, we can get the equations for the stress tensor $(\tau_{ij}^\eps)$. Here due to the fact that $\tau_{ij}=\tau_{ji}$, we only give the equations of $\tau_{ij}$ for $i\leqslant j$. More precisely, we find the diagonal elements of stress tensor $(\tau_{ii}^\eps)$ $(1\leqslant i\leqslant 3)$ satisfies
\beq\label{Oldroyd-B-Scaled2}
\left\{
\begin{split}
&\eps\left(\p_t+u^\eps\cdot\nabla_x+v^\eps\p_y\right)\tau_{11}^\eps+\tau_{11}^\eps+\left[\eps\tau_{12}^\eps(\p_{x_1}u_2^\eps-\p_{x_2}u_1^\eps)-\tau_{13}^\eps(\p_yu_1^\eps-\eps^2\p_{x_1}v^\eps)\right]\\
&+b\left[2\eps\tau_{11}^\eps\p_{x_1}u_1^\eps+\eps\tau_{12}^\eps(\p_{x_1}u_2^\eps+\p_{x_2}u_1^\eps)+\tau_{13}^\eps(\p_yu_1^\eps+\eps^2\p_{x_1}v^\eps)\right]
=2(1-\theta)\eps\p_{x_1}u_1^\eps,\\
&\eps\left(\p_t+u^\eps\cdot\nabla_x+v^\eps\p_y\right)\tau_{22}^\eps+\tau_{22}^\eps+\left[\eps\tau_{12}^\eps(\p_{x_2}u_1^\eps-\p_{x_1}u_2^\eps)-\tau_{23}^\eps(\p_yu_2^\eps-\eps^2\p_{x_2}v^\eps)\right]\\
&+b\left[\eps\tau_{12}^\eps(\p_{x_2}u_1^\eps+\p_{x_1}u_2^\eps)+2\eps\tau_{22}^\eps\p_{x_2}u_2^\eps+\tau_{23}^\eps(\p_yu_2^\eps+\eps^2\p_{x_2}v^\eps)\right]
=2(1-\theta)\eps\p_{x_2}u_2^\eps,\\
&\eps\left(\p_t+u^\eps\cdot\nabla_x+v^\eps\p_y\right)\tau_{33}^\eps+\tau_{33}^\eps+\left[\tau_{13}^\eps(\p_yu_1^\eps-\eps^2\p_{x_1}v^\eps)+\tau_{23}^\eps(\p_yu_2^\eps-\eps^2\p_{x_2}v^\eps)\right]\\
&+b\left[\tau_{13}^\eps(\p_yu_1^\eps+\eps^2\p_{x_1}v^\eps)+\tau_{23}^\eps(\p_yu_2^\eps+\eps^2\p_{x_2}v^\eps)+2\eps\tau_{33}^\eps\p_yv^\eps\right]
=2(1-\theta)\eps\p_yv^\eps.
\end{split}
\right.
\eeq
The first and the second diagonal elements $(\tau_{ij}^\eps)$ $(1\leqslant i<j\leqslant 3)$ above the main diagonal elements satisfies
\beq\label{Oldroyd-B-Scaled3}
\small
\left\{
\begin{split}
&\eps\left(\p_t+u^\eps\cdot\nabla_x+v^\eps\p_y\right)\tau_{12}^\eps+\tau_{12}^\eps+\f12b\eps\left[2\tau_{12}^\eps(\p_{x_1}u_1^\eps+\p_{x_2}u_2^\eps)+(\tau_{11}^\eps+\tau_{22}^\eps)(\p_{x_1}u_2^\eps+\p_{x_2}u_1^\eps)\right] \\
&+\f12\eps(\tau_{11}^\eps-\tau_{22}^\eps)(\p_{x_2}u_1^\eps-\p_{x_1}u_2^\eps)-\f12\tau_{13}^\eps(\p_yu_2^\eps-\eps^2\p_{x_2}v^\eps)-\f12\tau_{32}^\eps(\p_yu_1^\eps-\eps^2\p_{x_1}v^\eps)\\
&+\f12b\left[\tau_{13}^\eps(\p_yu_2^\eps+\eps^2\p_{x_2}v^\eps)+\tau_{32}^\eps(\p_yu_1^\eps+\eps^2\p_{x_1}v^\eps)\right]=(1-\theta)\eps(\p_{x_1}u_2^\eps+\p_{x_2}u_1^\eps),\\
&\eps\left(\p_t+u^\eps\cdot\nabla_x+v^\eps\p_y\right)\tau_{13}^\eps+\tau_{13}^\eps+\f12b\eps\left[2\tau_{13}^\eps(\p_{x_1}u_1^\eps+\p_yv^\eps)+\tau_{23}^\eps(\p_{x_1}u_2^\eps+\p_{x_2}u_1^\eps)\right]\\
&+\f12(\tau_{11}^\eps-\tau_{33}^\eps)(\p_yu_1^\eps-\eps^2\p_{x_1}v^\eps)+\f12\tau_{12}^\eps(\p_yu_2^\eps-\eps^2\p_{x_2}v^\eps)-\f12\eps\tau_{23}^\eps(\p_{x_2}u_1^\eps-\p_{x_1}u_2^\eps)\\
&+\f12b\left[(\tau_{11}^\eps+\tau_{33}^\eps)(\p_yu_1^\eps+\eps^2\p_{x_1}v^\eps)+\tau_{12}^\eps(\p_yu_2^\eps+\eps^2\p_{x_2}v^\eps)\right]=(1-\theta)(\p_yu_1^\eps+\eps^2\p_{x_1}v^\eps),\\
&\eps\left(\p_t+u^\eps\cdot\nabla_x+v^\eps\p_y\right)\tau_{23}^\eps+\tau_{23}^\eps+\f12b\eps\left[2\tau_{23}^\eps(\p_{x_2}u_2^\eps+\p_yv^\eps)+\tau_{13}^\eps(\p_{x_1}u_2^\eps+\p_{x_2}u_1^\eps)\right]\\
&+\f12(\tau_{22}^\eps-\tau_{33}^\eps)(\p_yu_2^\eps-\eps^2\p_{x_2}v^\eps)+\f12\tau_{21}^\eps(\p_yu_1^\eps-\eps^2\p_{x_1}v^\eps)-\f12\eps\tau_{13}^\eps(\p_{x_1}u_2^\eps-\p_{x_2}u_1^\eps)\\
&+\f12b\left[(\tau_{22}^\eps+\tau_{33}^\eps)(\p_yu_2^\eps+\eps^2\p_{x_2}v^\eps)+\tau_{21}^\eps(\p_yu_1^\eps+\eps^2\p_{x_1}v^\eps)\right]=(1-\theta)(\p_{y}u_2^\eps+\eps^2\p_{x_2}v^\eps).
\end{split}
\right.
\eeq
Formally, if we take $\eps\rightarrow0$, then the equations \eqref{Oldroyd-B-Scaled1} is reduced to
\beq\label{Oldroyd-B-Prandtl1}
\left\{
\begin{split}
&\big(\p_t+u\cdot\nabla_x+v\p_y-\theta\p_y^2\big)u+\nabla_x p=\begin{pmatrix}
 \p_y\tau_{13}\\[6pt]
 \p_y\tau_{23}
\end{pmatrix},\\
&\p_y p=0,\\
&\nabla_x\cdot u+\p_yv=0, \quad \int_{\mathbb{T}}v(t,x,y)\mathrm{d}y=0.
\end{split}
\right.
\eeq
The equations \eqref{Oldroyd-B-Scaled2}--\eqref{Oldroyd-B-Scaled3} is reduced to
\beq\label{Oldroyd-B-Prandtl2}
\left\{
\begin{split}
&(b-1)\tau_{13}\p_yu_1+\tau_{11}=0,\\
&(b-1)\tau_{23}\p_yu_2+\tau_{22}=0,\\
&(b+1)(\tau_{13}\p_yu_1+\tau_{23}\p_yu_2)+\tau_{33}=0,\\
&(b-1)(\tau_{13}\p_yu_2+\tau_{23}\p_yu_1)+2\tau_{12}=0,\\
&b(\tau_{11}+\tau_{33})\p_yu_1+(\tau_{11}-\tau_{33})\p_yu_1+(b+1)\tau_{12}\p_yu_2+2\tau_{13}=2(1-\theta)\p_yu_1,\\
&b(\tau_{22}+\tau_{33})\p_yu_2+(\tau_{22}-\tau_{33})\p_yu_2+(b+1)\tau_{21}\p_yu_1+2\tau_{23}=2(1-\theta)\p_yu_2.
\end{split}
\right.
\eeq
By virtue of \eqref{Oldroyd-B-Prandtl2}, we have that
\beq\label{tau13+tau23}
\begin{split}
&\tau_{23}=\frac{(1-\theta)\p_yu_2}{1+(1-b^2)((\p_yu_1)^2+(\p_yu_2)^2)},\\
&\tau_{13}=\frac{(1-\theta)\p_yu_1}{1+(1-b^2)((\p_yu_1)^2+(\p_yu_2)^2)}.
\end{split}
\eeq
The detail derivation of \eqref{tau13+tau23} can be found in the \ref{Appendix A}. Thus denote by $\sigma=1-b^2$, the equations \eqref{Oldroyd-B-Prandtl1} is converted to the following:
\beq\label{Oldroyd-B-Prandtl}
\left\{
\begin{split}
&\big(\p_t+u\cdot\nabla_x+v\p_y-\theta\p_y^2\big)u+\nabla_x p=(1-\theta)\p_y\begin{pmatrix}
 \dfrac{\p_yu_1}{1+\sigma((\p_yu_1)^2+(\p_yu_2)^2)}\\[6pt]
 \dfrac{\p_yu_2}{1+\sigma((\p_yu_1)^2+(\p_yu_2)^2)}
\end{pmatrix},\\
&\p_y p=0,\\
&\nabla_x\cdot u+\p_yv=0, \quad \int_{\mathbb{T}}v(t,x,y)\mathrm{d}y=0.
\end{split}
\right.
\eeq

The goal of this paper is to establish the global well-posedness of the equations \eqref{Oldroyd-B-Prandtl} with small analytic-in-$x$ initial data and then justify the limit from the equations \eqref{Oldroyd-B-Scaled1}--\eqref{Oldroyd-B-Scaled3} to the equations \eqref{Oldroyd-B-Prandtl1}--\eqref{Oldroyd-B-Prandtl2} as $\eps\rightarrow0$.

Now we state our main results. The first result is the global well-posedness of the hydrostatic Oldroyd-B equations \eqref{Oldroyd-B-Prandtl}.
\begin{thm}\label{th1.1}
Let $a>0$ and $s_1>\f32,s_2>\f12$. Then there exists a constant $c_1>0$ sufficiently small such that the following conclusion holds. If the initial data $u_0$ in \eqref{Oldroyd-B-Prandtl} satisfies
\beq\label{initial}
\normm{e^{a\langle D_x\rangle}u_0}_{H^{s_1,s_2}}+\normm{e^{a\langle D_x\rangle}\p_yu_0}_{H^{s_1,s_2}}\leqslant c_1a
\eeq
and the compatibility condition 
\beq\label{comp++1}
\int_{\mathbb{T}}u_0(x,y)\mathrm{d}y=0,
\eeq
then the hydrostatic Oldroyd-B equations \eqref{Oldroyd-B-Prandtl} admit a unique global-in-time solution $u$ satisfying that
\begin{multline}\label{energy inequality}
\normm{e^{\frak{K}t}e^{\frac{a}{2} \langle D_x \rangle}u}_{L_t^\infty H^{s_1,s_2}}+\normm{e^{\frak{K}t}e^{\frac{a}{2} \langle D_x \rangle}\p_yu}_{L_t^\infty H^{s_1,s_2}}
+\normm{e^{\frak{K}t}e^{\frac{a}{2} \langle D_x \rangle}\p_yu}_{L_t^2H^{s_1,s_2}}\\
+\normm{e^{\frak{K}t}e^{\frac{a}{2} \langle D_x \rangle}\p_y^2u}_{L_t^2H^{s_1,s_2}}\leqslant 100\left(\normm{e^{a\langle D_x \rangle}u_0}_{H^{s_1,s_2}}+\normm{e^{a\langle D_x \rangle}\p_yu_0}_{H^{s_1,s_2}}\right),
\end{multline}
where $e^{\frac{a}{2} \langle D_x \rangle}$ is a Fourier multiplier with symbol $e^{\frac{a}{2}(1+|\xi|)}$ and the constant $\frak{K}$ is determined by the Poinca\'re inequality on strip $\Omega$ (see \eqref{Poincare}). 
\end{thm}

\begin{remark}
We explain the compatibility condition in \eqref{comp++1}. By integrating $\partial_xu+\partial_yv=0$ over $\mathbb{T}$ with respect to vertical variable,
\beqo
\partial_x\int_\mathbb{T}u(t,x,y)\mathrm{d}y=0,
\eeqo
which together with the fact that: $u(t,x,y)\rightarrow0$ as $\abs{x}\rightarrow+\infty$, ensures that 
\beqo
\int_\mathbb{T} u(t,x,y)\mathrm{d}y=0.
\eeqo
\end{remark}

The second result is the justification for the hydrostatic limit from the equations \eqref{Oldroyd-B-Scaled1}--\eqref{Oldroyd-B-Scaled3} to the equations \eqref{Oldroyd-B-Prandtl1}--\eqref{Oldroyd-B-Prandtl2} as $\eps\rightarrow0$. 
\begin{thm}\label{th1.2}
Given $a>0$ and $s_1>\f52,s_2>\f32$. Suppose that the assumptions in Theorem \ref{th1.1} holds. Let $(u^\eps,v^\eps,p^\eps)$ and $(\tau_{ij}^\eps)$ $(1\leqslant i\leqslant j\leqslant 3)$ be the smooth solution of re-scaled Oldroyd-B equations \eqref{Oldroyd-B-Scaled1}--\eqref{Oldroyd-B-Scaled3} satisfying that
\begin{align}\label{assump-eps}
&\normm{e^{\frac{a}{2} \langle D_x \rangle}(u^\eps,\eps v^\eps)}_{L_t^\infty H^{s_1,s_2}}+\normm{e^{\frac{a}{2} \langle D_x \rangle}\sqrt{\eps}(\tau_{ij}^\eps)}_{L_t^\infty H^{s_1,s_2}}
+\normm{e^{\frac{a}{2} \langle D_x \rangle}(\eps\nabla_x,\p_y)(u^\eps,\eps v^\eps)}_{L_t^2H^{s_1,s_2}} \nonumber\\
&+\normm{e^{\frac{a}{2} \langle D_x \rangle}(\eps\nabla_x,\p_y)(u^\eps,\eps v^\eps)}_{L_t^1H^{s_1,s_2}}+\normm{e^{\frac{a}{2} \langle D_x \rangle}(\tau_{ij}^\eps)}_{L_t^2H^{s_1,s_2}}\leqslant 100c_1a,
\end{align}
Then it holds that
\begin{align}\label{hydrostaticlimit}
\normm{e^{\frac{a}{4} \langle D_x \rangle}(u^R,\eps v^R)}_{L_t^\infty H^{s_1-1,s_2-1}}+\normm{e^{\frac{a}{2} \langle D_x \rangle}\sqrt{\eps}(\tau_{ij}^R)}_{L_t^\infty H^{s_1-1,s_2-1}}\leqslant C\eps,	
\end{align}
where $(u^R,v^R)=(u^\eps-u,v^\eps-v), \tau_{ij}^R=\tau_{ij}^\eps-\tau_{ij}$ and the constant $C$ is independent of $\eps$.
\end{thm}

\begin{remark}
We remark that the initial condition for $\tau$ in the equations \eqref{Oldroyd-B-Prandtl1}--\eqref{Oldroyd-B-Prandtl2} is determined by the velocity field. Indeed, denote by $u_0=(u_0^1,u_0^2)$, then taking $t=0$ in \eqref{tau13+tau23}, we find that
\beq\label{tau13+tau23-initial}
\begin{split}
&\tau_{23}|_{t=0}=\frac{(1-\theta)\p_yu_0^2}{1+(1-b^2)((\p_yu_0^1)^2+(\p_yu_0^2)^2)},\\
&\tau_{13}|_{t=0}=\frac{(1-\theta)\p_yu_0^1}{1+(1-b^2)((\p_yu_0^1)^2+(\p_yu_0^2)^2)}.
\end{split}
\eeq	
Then by virtue of \eqref{Oldroyd-B-Prandtl2}$_1$--\eqref{Oldroyd-B-Prandtl2}$_4$, we have that 
\beq\label{tau11+tau22+tau33+tau12-initial}
\begin{split}
&\tau_{11}|_{t=0}=-(b-1)\tau_{13}|_{t=0}\p_yu_0^1,\\
&\tau_{22}|_{t=0}=-(b-1)\tau_{23}|_{t=0}\p_yu_0^2,\\
&\tau_{33}|_{t=0}=-(b+1)\left(\tau_{13}|_{t=0}\p_yu_0^1+\tau_{23}|_{t=0}\p_yu_0^2\right),\\
&\tau_{12}|_{t=0}=-\f12(b-1)\left(\tau_{13}|_{t=0}\p_yu_0^2+\tau_{23}|_{t=0}\p_yu_0^1\right).
\end{split}
\eeq
Here for the sake of a short presentation, we assume that $\tau_{ij}^\eps$ and $\tau_{ij}$ share the same initial data
\[
\tau_{ij}^0=\tau_{ij}|_{t=0},\quad 1\leqslant i,j\leqslant3.
\]
\end{remark}

We end this introduction with some notations which will be used throughout our paper. We denote by $L^q_{x}L^r_{y}$ the anisotropic Lebesgue space $L^q(\R^2;L^r(\mathbb{T}))$ and $H^{s_1,s_2}(\Omega)$ the standard anisotropic Sobolev spaces. Given $1\leqslant p\leqslant\infty$. Assume that $f(t)\in L^1_{\mathrm{loc}}(\R_+)$ be a nonnegative function, we define the time-space Bochner norm with weight $f(t)$ as follows:
\begin{align}
\normm{a}_{{L}^p_{t,f}(H^{s_1,s_2})}:=\left(\int_0^t f(t')\normm{a(t')}_{H^{s_1,s_2}}^p dt'\right)^{\f1p} \label{1.9} 
\end{align}
with the usual change if $p=\infty$.

\section{Global well-posedness of the equations \eqref{Oldroyd-B-Prandtl}: Proof of Theorem \ref{th1.1}}
In this section, we present the proof of Theorem \ref{th1.1}, that is, the global well-posedness of hydrostatic Oldroyd-B equations \eqref{Oldroyd-B-Prandtl}. For the sake of brevity, we only establish the {\it a prior} estimate in analytic space which is the core of the proof.

To characterize the evolution of the analytic band of $u$, we introduce
\begin{align}\label{theta}
\left\{\begin{aligned}
&\dot{\eta_1}(t)=\normm{\p_yu_\Psi(t)}_{H^{s_1,s_2}}\\
&\dot{\eta_2}(t)=\normm{\p_y^2u_{\Psi}(t)}_{H^{s_1,s_2}},\\
&\eta_1|_{t=0}=\eta_2|_{t=0}=0,
\end{aligned}
\right.
\quad\eta(t):=\eta_1(t)+\eta_2(t).\quad
\end{align}
Then to establish the {\it a prior} estimate in analytic space in the framework of Fourier analysis, we define the following weighted Fourier multiplier:
\beq\label{upsi}
u_{\Psi}:=\cF^{-1}(e^{\Psi}\widehat{u}(t,\xi,k)),
\eeq
where the phase function $\Psi$ satisfies
\beq\label{phase}
\Psi=(a-\lambda\eta(t))\left(1+\abs{\xi}\right), 
\eeq
the large constant $\lambda$ will be determined later.

We shall apply the standard bootstrapping argument (See \cite[Proposition 1.2.1]{Tao}) in our presentation. To this end, we denote by
\beq\label{time-theta} 
T^\star:=
\sup\left\{t>0\mid\eta(t)<\frac{a}{\lambda}, \quad \normm{(u_{\Psi},\p_yu_{\Psi})}_{L_t^\infty H^{s_1,s_2}}^2<\min\left\{\eps_1,\frac{1}{16C_1}\right\}\right\}, 
\eeq 
where the small constant $\eps_1$ (See \eqref{eps1-1} and \eqref{eps1-2}) and the large constant $C_1$ will be determined in Lemma \ref{nonlinear}. By virtue of \eqref{phase}, for $t<T^\star$, there holds the convex inequality:
\[
\Psi(t,\xi_1)\leqslant\Psi(t,\xi_1-\xi_2)+\Psi(t,\xi_2) \quad \Forall \xi_1,\xi_2\in\R.
\]
We shall show that $T^\star=+\infty$ in the following by the energy method.

\subsection{Proof of Theorem \ref{th1.1}: Estimate on $u$}
We re-write the equations for $u$ in \eqref{Oldroyd-B-Prandtl} as follows:
\begin{align}\label{HOB-u}
\p_tu+u\cdot\nabla_xu+v\p_yu-\p_y^2u+\nabla_xp=(1-\theta)\mathcal F,
\end{align}
where
\beq\label{FG1G2}
\begin{split}
&\mathcal F=\p_y^2u\mathcal G_1(\p_yu)-2\sigma\p_yu(\p_yu\cdot\p_y^2u)\mathcal G_2(\p_yu)-2\sigma\p_yu(\p_yu\cdot\p_y^2u),\\
&\mathcal G_1(\p_yu)=\frac{1}{1+\sigma((\p_yu_1)^2+(\p_yu_2)^2)}-1=\frac{1}{1+\sigma|\p_yu|^2}-1,\\
&\mathcal G_2(\p_yu)=\frac{1}{(1+\sigma((\p_yu_1)^2+(\p_yu_2)^2))^2}-1=\frac{1}{(1+\sigma|\p_yu|^2)^2}-1.	
\end{split}
\eeq
Applying the Fourier multiplier \eqref{upsi} to the equations \eqref{HOB-u}, we get
\begin{equation}\label{upsi-equ}
\begin{split}
\p_tu_\Psi+\lambda\dot{\eta_1}(t)\langle D_x \rangle u_\Psi+(u\cdot\nabla_xu)_\Psi+(v\p_yu)_\Psi-\p_y^2u_\Psi+\nabla_xp_{\Psi}=(1-\theta)\mathcal F_{\Psi}.	
\end{split}
\end{equation}
Notice that the Poincar\'e inequality implies that
\beq\label{Poincare}
\int_{\mathbb{T}}u(t,x,y)\mathrm{d}y=0\Longrightarrow\frak{K}\normm{u_\Psi}_{H^{s_1,s_2}}^2\leqslant \frac12\normm{\p_yu_\Psi}_{H^{s_1,s_2}}^2,
\eeq
where we can choose $\frak{K}<\frac{1}{2}$. Then performing the standard energy method for the equation \eqref{upsi-equ} and using the divergence-free condition, we obtain that
\beq\label{diff-u}
\begin{split}
&\f12\f{d}{dt}\normm{e^{\frak{K}t}u_\Psi(t)}_{H^{s_1,s_2}}^2+\lambda\dot{\eta_1}(t)\normm{e^{\frak{K}t}u_\Psi}_{H^{s_1+\f12,s_2}}^2+\f12\normm{e^{\frak{K}t}\p_yu_\Psi}_{H^{s_1,s_2}}^2\\
\leqslant&\abs{\inner{e^{\frak{K}t}(u\cdot\nabla_xu)_\Psi, e^{\frak{K}t}u_\Psi}_{H^{s_1,s_2}}}+\abs{\inner{e^{\frak{K}t}(v\p_yu)_\Psi, e^{\frak{K}t}u_\Psi}_{H^{s_1,s_2}}}\\
&+(1-\theta)\abs{\inner{e^{\frak{K}t}\mathcal F_{\Psi}, e^{\frak{K}t}u_\Psi}_{H^{s_1,s_2}}}.
\end{split}
\eeq
Now integrating the resulting inequality \eqref{diff-u} over $[0,t]$ with $t<T^*$, it holds that
\beq\label{inte-u}
\begin{split}
&\f12\normm{e^{\frak{K}t}u_\Psi(t)}_{L_t^\infty H^{s_1,s_2}}^2+\f12\normm{e^{\frak{K}t}\p_yu_\Psi}_{L_t^2H^{s_1,s_2}}^2+\lambda\normm{e^{\frak{K}t}u_\Psi}_{{L}^2_{t,\dot{\eta_1}(t)}(H^{s_1+\f12,s_2})}^2\\
\leqslant & \f12\normm{e^{a\langle D_x\rangle}u_0}_{H^{s_1,s_2}}^2+B_1+B_2+(1-\theta)(B_3+B_4+B_5),
\end{split}
\eeq
where
\begin{align*}
B_1&=\int_0^t\abs{\inner{e^{\frak{K}t^{\prime}}(u\cdot\nabla_xu)_\Psi, e^{\frak{K}t^{\prime}}u_\Psi}_{H^{s_1,s_2}}}\mathrm{d}t',\\
B_2&=\int_0^t\abs{\inner{e^{\frak{K}t^{\prime}}(v\p_yu)_\Psi, e^{\frak{K}t^{\prime}}u_\Psi}_{H^{s_1,s_2}}}\mathrm{d}t',\\
B_3&=\int_0^t\abs{\inner{e^{\frak{K}t^{\prime}}(\p_y^2u\mathcal G_1(\p_yu))_{\Psi}, e^{\frak{K}t^{\prime}}u_\Psi}_{H^{s_1,s_2}}}\mathrm{d}t',\\
B_4&=2\sigma\int_0^t\abs{\inner{e^{\frak{K}t^{\prime}}(\p_yu(\p_yu\cdot\p_y^2u)\mathcal G_2(\p_yu))_\Psi, e^{\frak{K}t^{\prime}}u_\Psi}_{H^{s_1,s_2}}}\mathrm{d}t',\\
B_5&=2\sigma\int_0^t\abs{\inner{e^{\frak{K}t^{\prime}}(\p_yu(\p_yu\cdot\p_y^2u))_\Psi, e^{\frak{K}t^{\prime}}u_\Psi}_{H^{s_1,s_2}}}\mathrm{d}t'.
\end{align*}
From the energy estimate for $u$, we find the higher derivative for $y$ variables appears in \eqref{inte-u} due to nonlinear term $\mathcal{F}$. This indicates that we have to combine the estimate on $\p_yu$.

\subsection{Proof of Theorem \ref{th1.1}: Estimate on $\partial_yu$}
Applying the operator $\p_y$ to \eqref{HOB-u}, we find $\p_yu$ satisfies the following:
\begin{align}\label{HOB-pyu}
\p_t\p_yu+u\cdot\nabla_x\p_yu+v\p_y^2u+\p_yu\cdot\nabla_xu+\p_yv\p_yu-\p_y^3u=&(1-\theta)\p_y\mathcal F.
\end{align}
Then performing the standard energy method for the equation \eqref{HOB-pyu} and using the divergence-free condition, we obtain that
\beq\label{diff-pyu}
\begin{split}
&\f12\f{d}{dt}\normm{e^{\frak{K}t}\p_yu_\Psi(t)}_{H^{s_1,s_2}}^2+\lambda\dot{\eta_1}(t)\normm{e^{\frak{K}t}\p_yu_\Psi}_{H^{s_1+\f12,s_2}}^2+\f12\normm{e^{\frak{K}t}\p_y^2u_\Psi}_{H^{s_1,s_2}}^2\\
\leqslant&\abs{\inner{e^{\frak{K}t}(u\cdot\nabla_x\p_yu)_\Psi+e^{\frak{K}t}(\p_yu\cdot\nabla_xu)_\Psi, e^{\frak{K}t}\p_yu_\Psi}_{H^{s_1,s_2}}}\\
&+\abs{\inner{e^{\frak{K}t}(v\p_y^2u)_\Psi+e^{\frak{K}t}(\p_yv\p_yu)_\Psi, e^{\frak{K}t}\p_yu_\Psi}_{H^{s_1,s_2}}}\\
&+(1-\theta)\abs{\inner{e^{\frak{K}t}\p_y\mathcal F_{\Psi}, e^{\frak{K}t}\p_yu_\Psi}_{H^{s_1,s_2}}}.
\end{split}
\eeq
Now integrating the resulting inequality \eqref{diff-pyu} over $[0,t]$ with $t<T^*$, it holds that
\beq\label{inte-pyu}
\begin{split}
&\f12\normm{e^{\frak{K}t}\p_yu_\Psi(t)}_{L_t^\infty H^{s_1,s_2}}^2+\f12\normm{e^{\frak{K}t}\p_y^2u_\Psi}_{L_t^2H^{s_1,s_2}}^2+\lambda\normm{e^{\frak{K}t}\p_yu_\Psi}_{{L}^2_{t,\dot{\eta_2}(t)}(H^{s_1+\f12,s_2})}^2\\
\leqslant & \f12\normm{e^{a\langle D_x\rangle}\p_yu_0}_{H^{s_1,s_2}}^2+B_6+B_7+(1-\theta)(B_{8}+B_{9}+B_{10}),
\end{split}
\eeq
where
\begin{align*}
B_6&=\int_0^t\abs{\inner{e^{\frak{K}t}(u\cdot\nabla_x\p_yu)_\Psi+e^{\frak{K}t}(\p_yu\cdot\nabla_xu)_\Psi, e^{\frak{K}t}\p_yu_\Psi}_{H^{s_1,s_2}}}\mathrm{d}t',\\
B_7&=\int_0^t\abs{\inner{e^{\frak{K}t}(v\p_y^2u)_\Psi+e^{\frak{K}t}(\p_yv\p_yu)_\Psi, e^{\frak{K}t}\p_yu_\Psi}_{H^{s_1,s_2}}}\mathrm{d}t',\\
B_8&=\int_0^t\abs{\inner{e^{\frak{K}t}(\p_y^2u\mathcal G_1(\p_yu))_{\Psi}, e^{\frak{K}t}\p_y^2u_\Psi}_{H^{s_1,s_2}}}\mathrm{d}t',\\
B_9&=2\sigma\int_0^t\abs{\inner{e^{\frak{K}t^{\prime}}(\p_yu(\p_yu\cdot\p_y^2u)\mathcal G_2(\p_yu))_\Psi, e^{\frak{K}t^{\prime}}\p_y^2u_\Psi}_{H^{s_1,s_2}}}\mathrm{d}t',\\
B_{10}&=2\sigma\int_0^t\abs{\inner{e^{\frak{K}t^{\prime}}(\p_yu(\p_yu\cdot\p_y^2u))_\Psi, e^{\frak{K}t^{\prime}}\p_y^2u_\Psi}_{H^{s_1,s_2}}}\mathrm{d}t'.
\end{align*}

\subsection{Completing the proof of Theorem \ref{th1.1}}

Summing the energy inequality \eqref{inte-u} and \eqref{inte-pyu}, we get 
\beq\label{inte-u+pyu-1}
\begin{split}
&\f12\normm{e^{\frak{K}t}u_\Psi(t)}_{L_t^\infty H^{s_1,s_2}}^2+\f12\normm{e^{\frak{K}t}\p_yu_\Psi}_{L_t^2H^{s_1,s_2}}^2+\lambda\normm{e^{\frak{K}t}u_\Psi}_{{L}^2_{t,\dot{\eta_1}(t)}(H^{s_1+\f12,s_2})}^2\\
&+\f12\normm{e^{\frak{K}t}\p_yu_\Psi(t)}_{L_t^\infty H^{s_1,s_2}}^2+\f12\normm{e^{\frak{K}t}\p_y^2u_\Psi}_{L_t^2H^{s_1,s_2}}^2+\lambda\normm{e^{\frak{K}t}\p_yu_\Psi}_{{L}^2_{t,\dot{\eta_2}(t)}(H^{s_1+\f12,s_2})}^2\\
\leqslant & \f12\normm{e^{a\langle D_x\rangle}u_0}_{H^{s_1,s_2}}^2+\f12\normm{e^{a\langle D_x\rangle}\p_yu_0}_{H^{s_1,s_2}}^2+B_1+B_2+B_6+B_7\\
&+(1-\theta)(B_3+B_4+B_5+B_{8}+B_{9}+B_{10}).
\end{split}
\eeq
We deal with the right-hand-side of the above inequality by the following lemma. 
\begin{mylem}\label{nonlinear}
Suppose that $\nabla_x\cdot u+\p_yv=0$, then for $s_1>\f32,s_2>\f12$ and $t<T^\star$, there holds 
\begin{align*}
&B_1+B_2+B_6+B_7+(1-\theta)(B_3+B_4+B_5+B_{8}+B_{9}+B_{10})\\
\leqslant &C_1\left(\normm{e^{\frak{K}t}u_\Psi}_{{L}^2_{t,\dot{\eta_1}(t)}(H^{s_1+\f12,s_2})}^2+\normm{e^{\frak{K}t}\p_yu_\Psi}_{{L}^2_{t,\dot{\eta_2}(t)}(H^{s_1+\f12,s_2})}^2\right)\\
&+C_1\left(\normm{u_\Psi}_{{L}^\infty_{t}(H^{s_1,s_2})}^2+\normm{\p_yu_\Psi}_{{L}^\infty_{t}(H^{s_1,s_2})}^2+\normm{\p_yu_\Psi}_{{L}^\infty_{t}(H^{s_1,s_2})}^4++\normm{\p_yu_\Psi}_{{L}^\infty_{t}(H^{s_1,s_2})}^6\right)\\
&\qquad\times\left(\normm{e^{\frak{K}t}\p_yu_\Psi}_{L_t^2(H^{s_1,s_2})}^2+\normm{e^{\frak{K}t}\p_y^2u_\Psi}_{L_t^2(H^{s_1,s_2})}^2\right),
\end{align*}
where the constant $C_1$ depends on $s_1,s_2$, $\mathcal G_1(z)$ and $\mathcal G_1(z)$ (See \eqref{HOB-u}).
\end{mylem}

The proof of Lemma \eqref{nonlinear} will be presented in the end of this section. Now we get, by virtue of \eqref{inte-u}, \eqref{inte-pyu}, and Lemma \ref{nonlinear} that 
\beqo
\begin{split}
&\f12\normm{e^{\frak{K}t}u_\Psi(t)}_{L_t^\infty H^{s_1,s_2}}^2+\f12\normm{e^{\frak{K}t}\p_yu_\Psi}_{L_t^2H^{s_1,s_2}}^2+\lambda\normm{e^{\frak{K}t}u_\Psi}_{{L}^2_{t,\dot{\eta_1}(t)}(H^{s_1+\f12,s_2})}^2\\
&+\f12\normm{e^{\frak{K}t}\p_yu_\Psi(t)}_{L_t^\infty H^{s_1,s_2}}^2+\f12\normm{e^{\frak{K}t}\p_y^2u_\Psi}_{L_t^2H^{s_1,s_2}}^2+\lambda\normm{e^{\frak{K}t}\p_yu_\Psi}_{{L}^2_{t,\dot{\eta_2}(t)}(H^{s_1+\f12,s_2})}^2\\
\leqslant & \f12\normm{e^{a\langle D_x\rangle}u_0}_{H^{s_1,s_2}}^2+\f12\normm{e^{a\langle D_x\rangle}\p_yu_0}_{H^{s_1,s_2}}^2\\
&+C_1\left(\normm{e^{\frak{K}t}u_\Psi}_{{L}^2_{t,\dot{\eta_1}(t)}(H^{s_1+\f12,s_2})}^2+\normm{e^{\frak{K}t}\p_yu_\Psi}_{{L}^2_{t,\dot{\eta_2}(t)}(H^{s_1+\f12,s_2})}^2\right)\\
&+C_1\left(\normm{u_\Psi}_{{L}^\infty_{t}(H^{s_1,s_2})}^2+\normm{\p_yu_\Psi}_{{L}^\infty_{t}(H^{s_1,s_2})}^2+\normm{\p_yu_\Psi}_{{L}^\infty_{t}(H^{s_1,s_2})}^4++\normm{\p_yu_\Psi}_{{L}^\infty_{t}(H^{s_1,s_2})}^6\right)\\
&\quad\times\left(\normm{e^{\frak{K}t}\p_yu_\Psi}_{L_t^2(H^{s_1,s_2})}^2+\normm{e^{\frak{K}t}\p_y^2u_\Psi}_{L_t^2(H^{s_1,s_2})}^2\right).
\end{split}
\eeqo
Thus taking $\lambda=2C_1$ in the above inequality and using \eqref{time-theta}, we deduce
\begin{multline}\label{Inte-u+pyu}
\normm{e^{\frak{K}t}u_\Psi(t)}_{L_t^\infty H^{s_1,s_2}}^2+\normm{e^{\frak{K}t}\p_yu_\Psi}_{L_t^2H^{s_1,s_2}}^2+\normm{e^{\frak{K}t}\p_yu_\Psi(t)}_{L_t^\infty H^{s_1,s_2}}^2+\normm{e^{\frak{K}t}\p_y^2u_\Psi}_{L_t^2H^{s_1,s_2}}^2\\
\leqslant 2\left(\normm{e^{a\langle D_x\rangle}u_0}_{H^{s_1,s_2}}^2+\normm{e^{a\langle D_x\rangle}\p_yu_0}_{H^{s_1,s_2}}^2\right),
\end{multline}
for $t<T^\star$. In particular, it holds that
\beq
\begin{split}\label{the-esti}
\eta(t)= &\int_0^t\normm{\p_yu_\Psi(t')}_{H^{s_1,s_2}}+\normm{\p_y^2u_\Psi(t')}_{H^{s_1,s_2}}\mathrm{d}t'\\
\leqslant &\left(\int_0^te^{-2\frak{K}t'}\mathrm{d}t'\right)^{\f12}\left(\normm{e^{\frak{K}t}\p_yu_\Psi}_{L_t^2H^{s_1,s_2}}+\normm{e^{\frak{K}t}\p_y^2u_\Psi}_{L_t^2H^{s_1,s_2}}\right)\\
\leqslant &2(2\frak{K})^{-\f12}\left(\normm{e^{a\langle D_x\rangle}u_0}_{H^{s_1,s_2}}+\normm{e^{a\langle D_x\rangle}\p_yu_0}_{H^{s_1,s_2}}\right). 
\end{split}
\eeq
Hence if we take $c_1$ in the initial data to be small such that
\beq\label{3.51}
\begin{split}
&2(2\frak{K})^{-\f12}\left(\normm{e^{a\langle D_x\rangle}u_0}_{H^{s_1,s_2}}+\normm{e^{a\langle D_x\rangle}\p_yu_0}_{H^{s_1,s_2}}\right)\leqslant 2(2\frak{K})^{-\f12}c_1a<\f{a}{2\lambda},\\
&\normm{e^{a\langle D_x\rangle}u_0}_{H^{s_1,s_2}}+\normm{e^{a\langle D_x\rangle}\p_yu_0}_{H^{s_1,s_2}}\leqslant c_1a<\min\left\{\frac{\eps_1}{2},\frac{1}{32C_1}\right\}.
\end{split}
\eeq
Then the bootstrapping argument shows that $T^\star=+\infty$. Finally, notice that the bootstrapping argument also implies that 
\beq\label{psilow}
\Psi(t)\geqslant\f12a\left(1+\abs{\xi}\right),
\eeq
for all $t>0$. Hence the global well-posedness of the equations \eqref{Oldroyd-B-Prandtl} follows from the standard regularization process and the estimate \eqref{Inte-u+pyu} implies that the energy inequality \eqref{energy inequality} in Theorem \ref{th1.1} holds. The proof of Theorem \ref{th1.1} is thus complete.

\subsection{Proof of Lemma \ref{nonlinear}}

In this subsection, we present the proof of Lemma \ref{nonlinear} and we set $s_1>\f32$ and $s_2>\f12$.

\noindent
$\bullet$ \underline{Estimate of $B_{1}$ and $B_{6}$.}\\ Notice that $s_1>\f32$ and $s_2>\f12$, we obtain from the product law (See Lemma \ref{lemma-product}) that
\begin{align*}
\abs{\inner{(u\cdot\nabla_xu)_\Psi,u_\Psi}_{H^{s_1,s_2}}}\leqslant &C\normm{(u\cdot\nabla_xu)_{\Psi}}_{H^{s_1-\f12,s_2}}\normm{u_{\Psi}}_{H^{s_1+\f12,s_2}}\\
\leqslant &C\normm{u_{\Psi}}_{H^{s_1-\f12,s_2}}\normm{u_{\Psi}}_{H^{s_1+\f12,s_2}}^2\\
\leqslant &C\normm{\p_yu_{\Psi}}_{H^{s_1,s_2}}\normm{u_{\Psi}}_{H^{s_1+\f12,s_2}}^2,
\end{align*}
where we also used the Poincar\'e inequality \eqref{Poincare} in the last inequality. Therefore we obtain
\begin{align}\label{B1}
B_1&=\int_0^t\abs{\inner{e^{\frak{K}t^{\prime}}(u\cdot\nabla_xu)_\Psi, e^{\frak{K}t^{\prime}}u_\Psi}_{H^{s_1,s_2}}}\mathrm{d}t'\leqslant C\normm{e^{\frak{K}t}u_\Psi}_{{L}^2_{t,\dot{\eta_1}(t)}(H^{s_1+\f12,s_2})}^2.
\end{align}
\noindent
Similarly, it holds that
\beq\label{B6}
\begin{split}
B_6\leqslant &\int_0^te^{2\frak{K}t^{\prime}}\normm{\p_y^2u_{\Psi}}_{H^{s_1,s_2}}\normm{\p_yu_{\Psi}}_{H^{s_1+\f12,s_2}}^2\mathrm{d}t'\\
&+\int_0^te^{2\frak{K}t^{\prime}}\normm{\p_y^2u_{\Psi}}_{H^{s_1,s_2}}^\f12\normm{\p_yu_{\Psi}}_{H^{s_1+\f12,s_2}}\normm{\p_yu_{\Psi}}_{H^{s_1,s_2}}^\f12\normm{u_{\Psi}}_{H^{s_1+\f12,s_2}}\mathrm{d}t'\\
\leqslant &C\left(\normm{e^{\frak{K}t}u_\Psi}_{{L}^2_{t,\dot{\eta_1}(t)}(H^{s_1+\f12,s_2})}^2+\normm{e^{\frak{K}t}\p_yu_\Psi}_{{L}^2_{t,\dot{\eta_2}(t)}(H^{s_1+\f12,s_2})}^2\right).
\end{split}
\eeq

\noindent
$\bullet$ \underline{Estimate of $B_{2}$ and $B_{7}$.}\\ Recalling 
\[
\int_{\mathbb{T}}v(t,x,y)\mathrm{d}y=0,
\]
thus by virtue of Lemma \ref{lemma-product} and Poincar\'e inequality, we have
\begin{align*}
\abs{\inner{(v\p_yu)_\Psi, u_\Psi}_{H^{s_1,s_2}}}\leqslant &C\normm{(v\p_yu)_{\Psi}}_{H^{s_1-\f12,s_2}}\normm{u_{\Psi}}_{H^{s_1+\f12,s_2}}\\
\leqslant &C\normm{\nabla_x\cdot u_{\Psi}}_{H^{s_1-\f12,s_2}}\normm{\p_yu_{\Psi}}_{H^{s_1-\f12,s_2}}\normm{u_{\Psi}}_{H^{s_1+\f12,s_2}}\\
\leqslant &C\normm{\p_yu_{\Psi}}_{H^{s_1,s_2}}\normm{u_{\Psi}}_{H^{s_1+\f12,s_2}}^2,
\end{align*}
where we also used the fact $\nabla_x\cdot u+\p_yv=0$ in the second inequality. Hence it yields that
\begin{align}\label{B2}
B_2=\int_0^t\abs{\inner{e^{\frak{K}t^{\prime}}(v\p_yu)_\Psi, e^{\frak{K}t^{\prime}}u_\Psi}_{H^{s_1,s_2}}}\mathrm{d}t'\leqslant C\normm{e^{\frak{K}t}u_\Psi}_{{L}^2_{t,\dot{\eta_1}(t)}(H^{s_1+\f12,s_2})}^2.
\end{align}
\noindent
Similarly, it holds that
\beq\label{B7}
\begin{split}
B_7\leqslant &\int_0^te^{2\frak{K}t^{\prime}}\normm{\p_y^2u_{\Psi}}_{H^{s_1,s_2}}\normm{\p_yu_{\Psi}}_{H^{s_1+\f12,s_2}}^2\mathrm{d}t'\\
&+\int_0^te^{2\frak{K}t^{\prime}}\normm{\p_y^2u_{\Psi}}_{H^{s_1,s_2}}^\f12\normm{\p_yu_{\Psi}}_{H^{s_1+\f12,s_2}}\normm{\p_yu_{\Psi}}_{H^{s_1,s_2}}^\f12\normm{u_{\Psi}}_{H^{s_1+\f12,s_2}}\mathrm{d}t'\\
\leqslant &C\left(\normm{e^{\frak{K}t}u_\Psi}_{{L}^2_{t,\dot{\eta_1}(t)}(H^{s_1+\f12,s_2})}^2+\normm{e^{\frak{K}t}\p_yu_\Psi}_{{L}^2_{t,\dot{\eta_2}(t)}(H^{s_1+\f12,s_2})}^2\right).
\end{split}
\eeq

\noindent
$\bullet$ \underline{Estimate on $B_{3}$ and $B_{8}$.}\\ Since the function $\mathcal G_1(z)=\frac{1}{1+\sigma z}-1$ is holomorphic in the region $\{z:|z|<\frac{1}{\sigma}\}$ of complex plane and $\mathcal G_1(0)=0$, there exists $\eps_1>0$ such that if 
\beq\label{eps1-1}
\normm{((\p_yu_1)^2+(\p_yu_2)^2)_{\Psi}}_{L_t^\infty H^{s_1,s_2}}<\normm{\p_yu_{\Psi}}_{L_t^\infty H^{s_1,s_2}}^2<\eps_1,
\eeq
then it follows from Lemma \ref{lemma-para-linearization-psi} that (setting $z=(\p_yu_1)^2+(\p_yu_2)^2$)
\beq\label{eps0}
\normm{(\mathcal G_1(\p_yu))_{\Psi}}_{H^{s_1,s_2}}\leqslant C\normm{((\p_yu_1)^2+(\p_yu_2)^2)_{\Psi}}_{H^{s_1,s_2}}\leqslant C\normm{\p_yu_{\Psi}}_{H^{s_1,s_2}}^2,
\eeq
where we also used Lemma \ref{lemma-product} in the second inequality. Thus we deduce from Lemma \ref{lemma-product} again that
\begin{align*}
\abs{\inner{(\p_y^2u\mathcal G_1(\p_yu))_{\Psi}, u_\Psi}_{H^{s_1,s_2}}}\leqslant C&\normm{\p_y^2u_{\Psi}}_{H^{s_1,s_2}}\normm{(\mathcal G_1(\p_yu))_{\Psi}}_{H^{s_1,s_2}}\normm{u_{\Psi}}_{H^{s_1,s_2}}\\
\leqslant C&\normm{\p_y^2u_{\Psi}}_{H^{s_1,s_2}}\normm{\p_yu_{\Psi}}_{H^{s_1,s_2}}^2\normm{u_{\Psi}}_{H^{s_1,s_2}}.
\end{align*}
Hence there holds
\beq\label{B3}
\begin{split}
B_3=&\int_0^t\abs{\inner{e^{\frak{K}t^{\prime}}(\p_y^2u\mathcal G_1(\p_yu))_{\Psi}, e^{\frak{K}t^{\prime}}u_\Psi}_{H^{s_1,s_2}}}\mathrm{d}t'\\
\leqslant &C\normm{u_\Psi}_{{L}^\infty_{t}(H^{s_1,s_2})}\normm{\p_yu_\Psi}_{{L}^\infty_{t}(H^{s_1,s_2})}\normm{e^{\frak{K}t}\p_yu_\Psi}_{L_t^2(H^{s_1,s_2})}\normm{e^{\frak{K}t}\p_y^2u_\Psi}_{L_t^2(H^{s_1,s_2})}\\
\leqslant &C\normm{u_\Psi}_{{L}^\infty_{t}(H^{s_1,s_2})}\normm{\p_yu_\Psi}_{{L}^\infty_{t}(H^{s_1,s_2})}\left(\normm{e^{\frak{K}t}\p_yu_\Psi}_{L_t^2(H^{s_1,s_2})}^2+\normm{e^{\frak{K}t}\p_y^2u_\Psi}_{L_t^2(H^{s_1,s_2})}^2\right).
\end{split}
\eeq
\noindent
Similarly, it holds that
\beq\label{B8}
\begin{split}
B_8=&\int_0^t\abs{\inner{e^{\frak{K}t^{\prime}}(\p_y^2u\mathcal G_1(\p_yu))_{\Psi}, e^{\frak{K}t^{\prime}}\p_y^2u_\Psi}_{H^{s_1,s_2}}}\mathrm{d}t'\\
\leqslant &C\normm{\p_yu_\Psi}_{{L}^\infty_{t}(H^{s_1,s_2})}^2\normm{e^{\frak{K}t}\p_y^2u_\Psi}_{L_t^2(H^{s_1,s_2})}^2.
\end{split}
\eeq

\noindent
$\bullet$ \underline{Estimate on $B_{4}$ and $B_{9}$.}\\ Following the similar argument in $B_{3}$, we obtain from $\mathcal G_2(0)=0$ that there exists $\eps_1>0$ such that if 
\beq\label{eps1-2}
\normm{((\p_yu_1)^2+(\p_yu_2)^2)_{\Psi}}_{L_t^\infty H^{s_1,s_2}}<\normm{\p_yu_{\Psi}}_{L_t^\infty H^{s_1,s_2}}^2<\eps_1,
\eeq
then Lemmas \ref{lemma-product}--\ref{lemma-para-linearization-psi} (setting $z=(\p_yu_1)^2+(\p_yu_2)^2$) imply that 
\beq\label{eps0}
\normm{(\mathcal G_2(\p_yu))_{\Psi}}_{H^{s_1,s_2}}\leqslant C\normm{((\p_yu_1)^2+(\p_yu_2)^2)_{\Psi}}_{H^{s_1,s_2}}\leqslant C\normm{\p_yu_{\Psi}}_{H^{s_1,s_2}}^2.
\eeq
Thus we deduce from Lemma \ref{lemma-product} again that
\begin{align*}
&\abs{\inner{(\p_yu(\p_yu\cdot\p_y^2u)\mathcal G_2(\p_yu))_\Psi, u_\Psi}_{H^{s_1,s_2}}}\\
\leqslant C&\normm{u_{\Psi}}_{H^{s_1,s_2}}\normm{\p_yu_{\Psi}}_{H^{s_1,s_2}}^2\normm{\p_y^2u_{\Psi}}_{H^{s_1,s_2}}\normm{(\mathcal G_2(\p_yu))_{\Psi}}_{H^{s_1,s_2}}\\
\leqslant C&\normm{u_{\Psi}}_{H^{s_1,s_2}}\normm{\p_yu_{\Psi}}_{H^{s_1,s_2}}^4\normm{\p_y^2u_{\Psi}}_{H^{s_1,s_2}}.
\end{align*}
Hence we have
\beq\label{B4}
\begin{split}
B_4=&2\sigma\int_0^t\abs{\inner{e^{\frak{K}t^{\prime}}(\p_yu(\p_yu\cdot\p_y^2u)\mathcal G_2(\p_yu))_\Psi, e^{\frak{K}t^{\prime}}u_\Psi}_{H^{s_1,s_2}}}\mathrm{d}t'\\
\leqslant &C\normm{u_\Psi}_{{L}^\infty_{t}(H^{s_1,s_2})}\normm{\p_yu_\Psi}_{{L}^\infty_{t}(H^{s_1,s_2})}^3\left(\normm{e^{\frak{K}t}\p_yu_\Psi}_{L_t^2(H^{s_1,s_2})}^2+\normm{e^{\frak{K}t}\p_y^2u_\Psi}_{L_t^2(H^{s_1,s_2})}^2\right).
\end{split}
\eeq
\noindent
Similarly, it holds that
\beq\label{B9}
\begin{split}
B_9=&2\sigma\int_0^t\abs{\inner{e^{\frak{K}t^{\prime}}(\p_yu(\p_yu\cdot\p_y^2u)\mathcal G_2(\p_yu))_\Psi, e^{\frak{K}t^{\prime}}\p_y^2u_\Psi}_{H^{s_1,s_2}}}\mathrm{d}t'\\
\leqslant &C\normm{\p_yu_\Psi}_{{L}^\infty_{t}(H^{s_1,s_2})}^4\normm{e^{\frak{K}t}\p_y^2u_\Psi}_{L_t^2(H^{s_1,s_2})}^2.
\end{split}
\eeq

\noindent
$\bullet$ \underline{Estimate on $B_{5}$ and $B_{10}$.}\\ The direct computation implies that
\beq\label{B5}
\begin{split}
B_5=&2\sigma\int_0^t\abs{\inner{e^{\frak{K}t^{\prime}}(\p_yu(\p_yu\cdot\p_y^2u))_\Psi, e^{\frak{K}t^{\prime}}u_\Psi}_{H^{s_1,s_2}}}\mathrm{d}t'\\
\leqslant &C\normm{u_\Psi}_{{L}^\infty_{t}(H^{s_1,s_2})}\normm{\p_yu_\Psi}_{{L}^\infty_{t}(H^{s_1,s_2})}\left(\normm{e^{\frak{K}t}\p_yu_\Psi}_{L_t^2(H^{s_1,s_2})}^2+\normm{e^{\frak{K}t}\p_y^2u_\Psi}_{L_t^2(H^{s_1,s_2})}^2\right).
\end{split}
\eeq
\noindent
Similarly, it holds that
\beq\label{B10}
\begin{split}
B_{10}=&2\sigma\int_0^t\abs{\inner{e^{\frak{K}t^{\prime}}(\p_yu(\p_yu\cdot\p_y^2u))_\Psi, e^{\frak{K}t^{\prime}}\p_y^2u_\Psi}_{H^{s_1,s_2}}}\mathrm{d}t'\\
\leqslant &C\normm{\p_yu_\Psi}_{{L}^\infty_{t}(H^{s_1,s_2})}^2\normm{e^{\frak{K}t}\p_y^2u_\Psi}_{L_t^2(H^{s_1,s_2})}^2.
\end{split}
\eeq

Now Lemma \ref{nonlinear} follows from \eqref{B1}, \eqref{B2}, \eqref{B3}, \eqref{B4}, \eqref{B5}, \eqref{B6}, \eqref{B7}, \eqref{B8}, \eqref{B9} and \eqref{B10}.

\section{Justification for the hydrostatic limit: Proof of Theorem \ref{th1.2}}
In this section, we present the proof of Theorem \ref{th1.2}, that is, the hydrostatic from the re-scaled Oldroyd-B equations \eqref{Oldroyd-B-Scaled1}--\eqref{Oldroyd-B-Scaled3} to the hydrostatic Oldroyd-B equations \eqref{Oldroyd-B-Prandtl1}--\eqref{Oldroyd-B-Prandtl2}. For this purpose, we write
\[
u^R=u^\eps-u,\quad v^R=v^\eps-v, \quad p^R=p^\eps-p, \quad \tau_{ij}^R=\tau_{ij}^\eps-\tau_{ij},\quad 1\leqslant i,j\leqslant3,
\]
where $u^R=(u_1^R,u_2^R)$ and $v^R$ are the tangential and normal error of velocities, respectively. Then recalling $\tau_{ij}^\eps=\tau_{ji}^\eps$ and $\tau_{ij}=\tau_{ji}$, we get, by a direct computation, that $(u^R,v^R,p^R)$ verifies
\beq
\left\{
\begin{split}
&\p_tu^R+u^R\cdot\nabla_x u+v^R\p_yu+u^\eps\cdot\nabla_x u^R+v^\eps\p_yu^R-\theta\Delta_{\eps}u^R+\nabla_x p^R\\
&\qquad\qquad=
\begin{pmatrix}
 \eps\p_{x_1}\tau_{11}^\eps+\eps\p_{x_2}\tau_{12}^\eps+\p_y\tau_{13}^R\\[6pt]
 \eps\p_{x_1}\tau_{21}^\eps+\eps\p_{x_2}\tau_{22}^\eps+\p_y\tau_{23}^R
\end{pmatrix}+\theta\eps^2\Delta_xu,\\
&\eps^2\left(\p_tv^R+u^R\cdot\nabla_xv+v^R\p_yv+u^\eps\cdot\nabla_xv^R+v^\eps\p_yv^R-\theta\Delta_{\eps}v^R\right)+\p_y p^R\\
&\qquad\qquad=\eps\left(\eps\p_{x_1}\tau_{31}^\eps+\eps\p_{x_2}\tau_{32}^\eps+\p_y\tau_{33}^\eps\right)
-\eps^2\left(\p_tv+u\cdot\nabla_xv+v\p_yv-\theta\Delta_{\eps}v\right),\\
&\nabla_x\cdot u^R+\p_yv^R=0,\quad\int_{\mathbb{T}}v^R(t,x,y)\mathrm{d}y=0.
\end{split}
\right.
\eeq
Similarly, we can get the error equations for the stress tensor. More precisely, the diagonal elements for the error of stress tensor $(\tau_{ii}^R)$ $(1\leqslant i\leqslant 3)$ satisfies
\beq
\left\{
\begin{split}
&\f12\eps\p_t\tau_{11}^R+\f12\tau_{11}^R=(1-\theta)\eps\p_{x_1}u_1^\eps-\f12\eps\left(\p_t\tau_{11}+u^\eps\cdot\nabla_x\tau_{11}^\eps+v^\eps\p_y\tau_{11}^\eps\right)\\
&\quad-b\eps\tau_{11}^\eps\p_{x_1}u_1^\eps-\f12\eps\tau_{12}^\eps\left[(b+1)\p_{x_1}u_2^\eps+(b-1)\p_{x_2}u_1^\eps\right]-\f12(b+1)\eps^2\tau_{13}^\eps\p_{x_1}v^\eps,\\
&\quad-\f12(b-1)\left(\tau_{13}^R\p_yu_1^\eps+\tau_{13}\p_yu_1^R\right),\\
&\f12\eps\p_t\tau_{22}^R+\f12\tau_{22}^R=(1-\theta)\eps\p_{x_2}u_2^\eps-\f12\eps\left(\p_t\tau_{22}+u^\eps\cdot\nabla_x\tau_{22}^\eps+v^\eps\p_y\tau_{22}^\eps\right)\\
&\quad-\f12\eps\tau_{12}^\eps\left[(b+1)\p_{x_2}u_1^\eps+(b-1)\p_{x_1}u_2^\eps\right]-b\eps\tau_{22}^\eps\p_{x_2}u_2^\eps-\f12(b+1)\eps^2\tau_{23}^\eps\p_{x_2}v^\eps,\\
&\quad-\f12(b-1)\left(\tau_{23}^R\p_yu_2^\eps+\tau_{23}\p_yu_2^R\right),\\
&\f12\eps\p_t\tau_{33}^R+\f12\tau_{33}^R=(1-\theta)\eps\p_yv^\eps-\f12\eps\left(\p_t\tau_{33}+u^\eps\cdot\nabla_x\tau_{33}^\eps+v^\eps\p_y\tau_{33}^\eps\right)\\
&\quad-\f12(b-1)\eps^2\left[\tau_{13}^\eps\p_{x_1}v^\eps+\tau_{23}^\eps\p_{x_2}v^\eps\right]-b\eps\tau_{33}^\eps\p_yv^\eps,\\
&\quad-\f12(b+1)\left(\tau_{13}^R\p_yu_1^\eps+\tau_{13}\p_yu_1^R+\tau_{23}^R\p_yu_2^\eps+\tau_{23}\p_yu_2^R\right).
\end{split}
\right.
\eeq
The first and the second diagonal elements for the error of stress tensor $(\tau_{ij}^R)$ satisfies
\beq
\small
\left\{
\begin{split}
&\eps\p_t\tau_{12}^R+\tau_{12}^R=(1-\theta)\eps(\p_{x_1}u_2^\eps+\p_{x_2}u_1^\eps)-\eps\left(\p_t\tau_{12}+u^\eps\cdot\nabla_x\tau_{12}^\eps+v^\eps\p_y\tau_{12}^\eps\right)\\
&\quad-\f12\eps\tau_{11}^\eps\left[(b+1)\p_{x_2}u_1^\eps+(b-1)\p_{x_1}u_2^\eps\right]-\f12\eps\tau_{22}^\eps\left[(b+1)\p_{x_1}u_2^\eps+(b-1)\p_{x_2}u_1^\eps\right]\\
&\quad-\f12(b+1)\eps^2\left[\tau_{13}^\eps\p_{x_2}v^\eps+\tau_{23}^\eps\p_{x_1}v^\eps\right]-b\eps\tau_{12}^\eps(\p_{x_1}u_1^\eps+\p_{x_2}u_2^\eps),\\
&\quad-\f12(b-1)\left(\tau_{13}^R\p_yu_2^\eps+\tau_{13}\p_yu_2^R+\tau_{23}^R\p_yu_1^\eps+\tau_{23}\p_yu_1^R\right),\\
&\eps\p_t\tau_{13}^R+\tau_{13}^R=(1-\theta)(\p_yu_1^R+\eps^2\p_{x_1}v^\eps)-\eps\left(\p_t\tau_{13}+u^\eps\cdot\nabla_x\tau_{13}^\eps+v^\eps\p_y\tau_{13}^\eps\right)\\
&\quad-b\eps\tau_{13}^\eps(\p_{x_1}u_1^\eps+\p_yv^\eps)-\f12\eps\tau_{23}^\eps\left[(b+1)\p_{x_1}u_2^\eps+(b-1)\p_{x_2}u_1^\eps\right]\\
&\quad-\f12(b-1)\eps^2\left[\tau_{11}^\eps\p_{x_1}v^\eps+\tau_{12}^\eps\p_{x_2}v^\eps\right]-\f12(b+1)\eps^2\tau_{33}^\eps\p_{x_1}v^\eps,\\
&\quad-\f12(b+1)\left(\tau_{11}^R\p_yu_1^\eps+\tau_{11}\p_yu_1^R+\tau_{12}^R\p_yu_2^\eps+\tau_{12}\p_yu_2^R\right)-\f12(b-1)\left(\tau_{33}^R\p_yu_1^\eps+\tau_{33}\p_yu_1^R\right),\\
&\eps\p_t\tau_{23}^R+\tau_{23}^R=(1-\theta)(\p_yu_2^R+\eps^2\p_{x_2}v^\eps)-\eps\left(\p_t\tau_{23}+u^\eps\cdot\nabla_x\tau_{23}^\eps+v^\eps\p_y\tau_{23}^\eps\right)\\
&\quad-b\eps\tau_{23}^\eps(\p_{x_2}u_2^\eps+\p_yv^\eps)-\f12\eps\tau_{13}^\eps\left[(b+1)\p_{x_2}u_1^\eps+(b-1)\p_{x_1}u_2^\eps\right]\\
&\quad-\f12(b-1)\eps^2\left[\tau_{22}^\eps\p_{x_2}v^\eps+\tau_{12}^\eps\p_{x_1}v^\eps\right]-\f12(b+1)\eps^2\tau_{33}^\eps\p_{x_2}v^\eps,\\
&\quad-\f12(b+1)\left(\tau_{22}^R\p_yu_2^\eps+\tau_{22}\p_yu_2^R+\tau_{12}^R\p_yu_1^\eps+\tau_{12}\p_yu_1^R\right)-\f12(b-1)\left(\tau_{33}^R\p_yu_2^\eps+\tau_{33}\p_yu_2^R\right).
\end{split}
\right.
\eeq
Then to establish the hydrostatic limit result, let $\widetilde\lambda\geqslant\lambda$ be a large constant which will be determined later, we define 
\beqo
f_{\Phi}:=\cF^{-1}(e^{\Phi}\widehat{f}(t,\xi,k)),\quad \Phi(t,\xi,k):=\f13(a-\widetilde\lambda\zeta(t))(1+|\xi|),
\eeqo
with
\begin{equation*}
\zeta(t):=\int_0^t\normm{e^{\frac{a}{2}\langle D_x\rangle}(\eps\nabla_x,\p_y)u^\eps(t')}_{H^{s_1,s_2}}+\normm{e^{\frac{a}{2}\langle D_x\rangle}\p_yu(t')}_{H^{s_1,s_2}}+\normm{e^{\frac{a}{2}\langle D_x\rangle}\p_y^2u(t')}_{H^{s_1,s_2}}\mathrm{d}t'.
\end{equation*}
Under the assumptions in Theorem \ref{th1.2}, we find that if $c_1$ enough small, then there holds (See also \eqref{psilow})
\beq\label{upperlow}
\f14a(1+|\xi|)\leqslant\Phi(t,\xi,k)\leqslant\f13\Psi(t,\xi,k)\leqslant\f13a(1+|\xi|)<\f12a(1+|\xi|).
\eeq
With the above preparations, mimic the proof of \eqref{inte-u+pyu-1}, we get
\beq\label{inte-diff}
\begin{split}
&\f12\normm{(u_{\Phi}^R,\eps v_{\Phi}^R)}_{L_t^\infty H^{s_1-1,s_2-1}}^2+\theta\normm{(\eps\nabla_x,\p_y)(u_{\Phi}^R,\eps v_{\Phi}^R)}_{L_t^2H^{s_1-1,s_2-1}}^2\\
&+\f12\normm{\sqrt{\eps}((\tau_{ij}^R)_{\Phi})}_{L_t^\infty H^{s_1-1,s_2-1}}^2+\f12\normm{((\tau_{ii}^R)_{\Phi})}_{L_t^2H^{s_1-1,s_2-1}}^2+\normm{((\tau_{ij}^R)_{i\neq j})_{\Phi}}_{L_t^2H^{s_1-1,s_2-1}}^2\\
&+\widetilde\lambda\left(\normm{(u_{\Phi}^R,\eps v_{\Phi}^R)}_{{L}^2_{t,\dot{\zeta}(t)}(H^{s_1-\f12,s_2-1})}^2+\normm{\sqrt{\eps}((\tau_{ij}^R)_{\Phi})}_{{L}^2_{t,\dot{\zeta}(t)}(H^{s_1-\f12,s_2-1})}^2\right)\\
\leqslant &K_1+\cdots+K_{16},
\end{split}
\eeq
where we used the notation $\tau_{ii}^R=(\tau_{11},\tau_{22},\tau_{33})$ and $(\tau_{ij}^R)_{i\neq j}=(\tau_{12},\tau_{13},\tau_{23})$ for brevity. Similar for $\tau_{ij}^R$.
\begin{align*}
K_1=&-\int_0^t\inner{(u^R\cdot\nabla_x u+v^R\p_yu)_\Phi, u_\Phi^R}_{H^{s_1-1,s_2-1}}\mathrm{d}t'\\
&-\eps^2\int_0^t\inner{(u^R\cdot\nabla_x v+v^R\p_yv)_\Phi, v_\Phi^R}_{H^{s_1-1,s_2-1}}\mathrm{d}t',\\
K_2=&-\int_0^t\inner{(u^\eps\cdot\nabla_x u^R+v^\eps\p_yu^R)_\Phi, u_\Phi^R}_{H^{s_1-1,s_2-1}}\mathrm{d}t'\\
&-\eps^2\int_0^t\inner{(u^\eps\cdot\nabla_x v^R+v^\eps\p_yv^R)_\Phi, v_\Phi^R}_{H^{s_1-1,s_2-1}}\mathrm{d}t',\\
K_3=&-\f12(b-1)\int_0^t\inner{\left(\tau_{13}^R\p_yu_1^\eps+\tau_{13}\p_yu_1^R\right)_\Phi, (\tau_{11}^R)_\Phi}_{H^{s_1-1,s_2-1}}\mathrm{d}t'\\
&-\f12(b-1)\int_0^t\inner{\left(\tau_{23}^R\p_yu_2^\eps+\tau_{23}\p_yu_2^R\right)_\Phi, (\tau_{22}^R)_\Phi}_{H^{s_1-1,s_2-1}}\mathrm{d}t'\\
&-\f12(b+1)\int_0^t\inner{\left(\tau_{13}^R\p_yu_1^\eps+\tau_{13}\p_yu_1^R+\tau_{23}^R\p_yu_2^\eps+\tau_{23}\p_yu_2^R\right)_\Phi, (\tau_{33}^R)_\Phi}_{H^{s_1-1,s_2-1}}\mathrm{d}t',\\
K_4=&-\f12(b-1)\int_0^t\inner{\left(\tau_{13}^R\p_yu_2^\eps+\tau_{13}\p_yu_2^R+\tau_{23}^R\p_yu_1^\eps+\tau_{23}\p_yu_1^R\right)_\Phi, (\tau_{12}^R)_\Phi}_{H^{s_1-1,s_2-1}}\mathrm{d}t'\\
&-\f12(b+1)\int_0^t\inner{\left(\tau_{11}^R\p_yu_1^\eps+\tau_{11}\p_yu_1^R+\tau_{12}^R\p_yu_2^\eps+\tau_{12}\p_yu_2^R\right)_\Phi, (\tau_{13}^R)_\Phi}_{H^{s_1-1,s_2-1}}\mathrm{d}t'\\
&-\f12(b-1)\int_0^t\inner{\left(\tau_{33}^R\p_yu_1^\eps+\tau_{33}\p_yu_1^R\right)_\Phi, (\tau_{13}^R)_\Phi}_{H^{s_1-1,s_2-1}}\mathrm{d}t'\\
&-\f12(b+1)\int_0^t\inner{\left(\tau_{22}^R\p_yu_2^\eps+\tau_{22}\p_yu_2^R+\tau_{12}^R\p_yu_1^\eps+\tau_{12}\p_yu_1^R\right)_\Phi, (\tau_{23}^R)_\Phi}_{H^{s_1-1,s_2-1}}\mathrm{d}t'\\
&-\f12(b-1)\int_0^t\inner{\left(\tau_{33}^R\p_yu_2^\eps+\tau_{33}\p_yu_2^R\right)_\Phi, (\tau_{23}^R)_\Phi}_{H^{s_1-1,s_2-1}}\mathrm{d}t',\\
K_5=&-\f12\eps\int_0^t\inner{\left(\p_t\tau_{11}+u^\eps\cdot\nabla_x\tau_{11}^\eps+v^\eps\p_y\tau_{11}^\eps\right)_\Phi, (\tau_{11}^R)_\Phi}_{H^{s_1-1,s_2-1}}\mathrm{d}t'\\
&-\f12\eps\int_0^t\inner{\left(\p_t\tau_{22}+u^\eps\cdot\nabla_x\tau_{22}^\eps+v^\eps\p_y\tau_{22}^\eps\right)_\Phi, (\tau_{22}^R)_\Phi}_{H^{s_1-1,s_2-1}}\mathrm{d}t'\\
&-\f12\eps\int_0^t\inner{\left(\p_t\tau_{33}+u^\eps\cdot\nabla_x\tau_{33}^\eps+v^\eps\p_y\tau_{33}^\eps\right)_\Phi, (\tau_{33}^R)_\Phi}_{H^{s_1-1,s_2-1}}\mathrm{d}t',\\
K_6=&-\eps\int_0^t\inner{\left(\p_t\tau_{12}+u^\eps\cdot\nabla_x\tau_{12}^\eps+v^\eps\p_y\tau_{12}^\eps\right)_\Phi, (\tau_{12}^R)_\Phi}_{H^{s_1-1,s_2-1}}\mathrm{d}t'\\
&-\eps\int_0^t\inner{\left(\p_t\tau_{13}+u^\eps\cdot\nabla_x\tau_{13}^\eps+v^\eps\p_y\tau_{13}^\eps\right)_\Phi, (\tau_{13}^R)_\Phi}_{H^{s_1-1,s_2-1}}\mathrm{d}t'\\
&-\eps\int_0^t\inner{\left(\p_t\tau_{23}+u^\eps\cdot\nabla_x\tau_{23}^\eps+v^\eps\p_y\tau_{23}^\eps\right)_\Phi, (\tau_{23}^R)_\Phi}_{H^{s_1-1,s_2-1}}\mathrm{d}t',\\
K_7=&-b\eps\int_0^t\inner{(\tau_{11}^\eps\p_{x_1}u_1^\eps)_\Phi, (\tau_{11}^R)_\Phi}_{H^{s_1-1,s_2-1}}+\inner{(\tau_{22}^\eps\p_{x_2}u_2^\eps)_\Phi, (\tau_{22}^R)_\Phi}_{H^{s_1-1,s_2-1}}\mathrm{d}t'\\
&-b\eps\int_0^t\inner{(\tau_{33}^\eps\p_yv^\eps)_\Phi, (\tau_{33}^R)_\Phi}_{H^{s_1-1,s_2-1}}\mathrm{d}t',\\
K_8=&-b\eps\int_0^t\inner{(\tau_{12}^\eps(\p_{x_1}u_1^\eps+\p_{x_2}u_2^\eps))_\Phi, (\tau_{12}^R)_\Phi}_{H^{s_1-1,s_2-1}}\mathrm{d}t'\\
&-b\eps\int_0^t\inner{(\tau_{13}^\eps(\p_{x_1}u_1^\eps+\p_yv^\eps))_\Phi, (\tau_{13}^R)_\Phi}_{H^{s_1-1,s_2-1}}\mathrm{d}t'\\
&-b\eps\int_0^t\inner{(\tau_{23}^\eps(\p_{x_2}u_2^\eps+\p_yv^\eps))_\Phi, (\tau_{23}^R)_\Phi}_{H^{s_1-1,s_2-1}}\mathrm{d}t',\\
K_9=&-\f12\eps\int_0^t\inner{(\tau_{12}^\eps\left[(b+1)\p_{x_1}u_2^\eps+(b-1)\p_{x_2}u_1^\eps\right])_\Phi, (\tau_{11}^R)_\Phi}_{H^{s_1-1,s_2-1}}\mathrm{d}t'\\
&-\f12\eps\int_0^t\inner{(\tau_{12}^\eps\left[(b+1)\p_{x_2}u_1^\eps+(b-1)\p_{x_1}u_2^\eps\right])_\Phi, (\tau_{22}^R)_\Phi}_{H^{s_1-1,s_2-1}}\mathrm{d}t',\\
K_{10}=&-\f12\eps\int_0^t\inner{(\tau_{11}^\eps\left[(b+1)\p_{x_2}u_1^\eps+(b-1)\p_{x_1}u_2^\eps\right])_\Phi, (\tau_{12}^R)_\Phi}_{H^{s_1-1,s_2-1}}\mathrm{d}t'\\
&-\f12\eps\int_0^t\inner{(\tau_{22}^\eps\left[(b+1)\p_{x_1}u_2^\eps+(b-1)\p_{x_2}u_1^\eps\right])_\Phi, (\tau_{12}^R)_\Phi}_{H^{s_1-1,s_2-1}}\mathrm{d}t'\\
&-\f12\eps\int_0^t\inner{(\tau_{23}^\eps\left[(b+1)\p_{x_1}u_2^\eps+(b-1)\p_{x_2}u_1^\eps\right])_\Phi, (\tau_{13}^R)_\Phi}_{H^{s_1-1,s_2-1}}\mathrm{d}t'\\
&-\f12\eps\int_0^t\inner{(\tau_{13}^\eps\left[(b+1)\p_{x_2}u_1^\eps+(b-1)\p_{x_1}u_2^\eps\right])_\Phi, (\tau_{23}^R)_\Phi}_{H^{s_1-1,s_2-1}}\mathrm{d}t',\\
K_{11}=&-\f12(b+1)\eps^2\int_0^t\inner{(\tau_{13}^\eps\p_{x_1}v^\eps)_\Phi, (\tau_{11}^R)_\Phi}_{H^{s_1-1,s_2-1}}+\inner{(\tau_{23}^\eps\p_{x_2}v^\eps)_\Phi, (\tau_{22}^R)_\Phi}_{H^{s_1-1,s_2-1}}\mathrm{d}t'\\
&-\f12(b-1)\eps^2\int_0^t\inner{(\tau_{13}^\eps\p_{x_1}v^\eps+\tau_{23}^\eps\p_{x_2}v^\eps)_\Phi, (\tau_{33}^R)_\Phi}_{H^{s_1-1,s_2-1}}\mathrm{d}t',\\
K_{12}=&-\f12(b+1)\eps^2\int_0^t\inner{(\tau_{13}^\eps\p_{x_2}v^\eps+\tau_{23}^\eps\p_{x_1}v^\eps)_\Phi, (\tau_{12}^R)_\Phi}_{H^{s_1-1,s_2-1}}\mathrm{d}t'\\
&-\f12(b-1)\eps^2\int_0^t\inner{(\tau_{11}^\eps\p_{x_1}v^\eps+\tau_{12}^\eps\p_{x_2}v^\eps)_\Phi, (\tau_{13}^R)_\Phi}_{H^{s_1-1,s_2-1}}\mathrm{d}t'\\
&-\f12(b+1)\eps^2\int_0^t\inner{(\tau_{33}^\eps\p_{x_1}v^\eps)_\Phi, (\tau_{13}^R)_\Phi}_{H^{s_1-1,s_2-1}}\mathrm{d}t'\\
&-\f12(b-1)\eps^2\int_0^t\inner{(\tau_{22}^\eps\p_{x_2}v^\eps+\tau_{12}^\eps\p_{x_1}v^\eps)_\Phi, (\tau_{23}^R)_\Phi}_{H^{s_1-1,s_2-1}}\mathrm{d}t'\\
&-\f12(b+1)\eps^2\int_0^t\inner{(\tau_{33}^\eps\p_{x_2}v^\eps)_\Phi, (\tau_{23}^R)_\Phi}_{H^{s_1-1,s_2-1}}\mathrm{d}t',\\
K_{13}=&\eps^2\int_0^t\inner{\theta\Delta_xu_\Phi, u^R_\Phi}_{H^{s_1-1,s_2-1}}-\inner{\left(\p_tv+u\cdot\nabla_xv+v\p_yv-\theta\Delta_{\eps}v\right)_\Phi, v^R_\Phi}_{H^{s_1-1,s_2-1}}\mathrm{d}t',\\
K_{14}=&\eps\int_0^t\inner{(\p_{x_1}\tau_{11}^\eps+\p_{x_2}\tau_{12}^\eps)_\Phi, (u_1^R)_\Phi}_{H^{s_1-1,s_2-1}}+\inner{(\p_{x_1}\tau_{12}^\eps+\p_{x_2}\tau_{22}^\eps)_\Phi, (u_2^R)_\Phi}_{H^{s_1-1,s_2-1}}\mathrm{d}t'\\
&+\eps\int_0^t\inner{(\eps\p_{x_1}\tau_{31}^\eps+\eps\p_{x_2}\tau_{32}^\eps+\p_y\tau_{33}^\eps)_\Phi, v^R_\Phi}_{H^{s_1-1,s_2-1}}\mathrm{d}t',\\
K_{15}=&(1-\theta)\eps\int_0^t\inner{(\p_{x_1}u_1^\eps)_\Phi, (\tau_{11}^R)_\Phi}_{H^{s_1-1,s_2-1}}+\inner{(\p_{x_2}u_2^\eps)_\Phi, (\tau_{22}^R)_\Phi}_{H^{s_1-1,s_2-1}}\mathrm{d}t'\\
&+(1-\theta)\eps\int_0^t\inner{\p_yv^\eps_\Phi, (\tau_{33}^R)_\Phi}_{H^{s_1-1,s_2-1}}+\inner{(\p_{x_1}u_2^\eps+\p_{x_2}u_1^\eps)_\Phi, (\tau_{12}^R)_\Phi}_{H^{s_1-1,s_2-1}}\mathrm{d}t'\\
&+(1-\theta)\eps^2\int_0^t\inner{\p_{x_1}v^\eps_\Phi, (\tau_{13}^R)_\Phi}_{H^{s_1-1,s_2-1}}+\inner{\p_{x_2}v^\eps_\Phi, (\tau_{23}^R)_\Phi}_{H^{s_1-1,s_2-1}}\mathrm{d}t',\\
K_{16}=&\int_0^t\inner{\p_y(\tau_{13}^R)_\Phi, (u_1^R)_\Phi}_{H^{s_1-1,s_2-1}}+\inner{\p_y(\tau_{23}^R)_\Phi, (u_2^R)_\Phi}_{H^{s_1-1,s_2-1}}\mathrm{d}t'\\
&+(1-\theta)\int_0^t\inner{\p_y(u_1^R)_\Phi, (\tau_{13}^R)_\Phi}_{H^{s_1-1,s_2-1}}+\inner{\p_y(u_2^R)_\Phi, (\tau_{23}^R)_\Phi}_{H^{s_1-1,s_2-1}}\mathrm{d}t'.
\end{align*}
To deal with the terms $K_1-K_{16}$, we have to estimate the time derivative of $u$ and $\tau$ (See $K_{5}$, $K_{6}$ and $K_{13}$). In fact, we have the following auxiliary lemma:
\begin{mylem}\label{time-tau}
Under the assumptions in Theorem \ref{th1.2}, it holds that
\begin{align*}
\normm{e^{\frak{K}t}e^{\frac{a}{2}\langle D_x\rangle}(\p_tu,\p_t\p_yu)}_{L_t^2H^{s_1-1,s_2-1}}+\normm{e^{\frac{a}{2}\langle D_x\rangle}\p_t\tau_{ij}}_{L_t^2 H^{s_1-1,s_2-1}}\leqslant C\normm{e^{a\langle D_x\rangle}(u_0,\p_yu_0)}_{H^{s_1,s_2}}
\end{align*}	
for any $1\leqslant i, j\leqslant3$, where the constant $C$ is independent of $\eps$.
\end{mylem}

Now with the help of Lemma \ref{time-tau}, the right hand of \eqref{inte-diff} can be estimated by the following lemma. More precisely, we have
\begin{mylem}\label{K1+cdots+K16}
Under the assumptions in Theorem \ref{th1.2}, it holds that
\begin{align*}
\sum_{i=1}^{16}K_i\leqslant &C\eps^2+\frac{3}{4}\theta\normm{(\eps\nabla_x,\p_y)(u_\Phi^R,\eps v_\Phi^R)}_{L_t^2H^{s_1-1,s_2-1}}^2+\frac{3}{8}\normm{((\tau_{ii}^R)_{\Phi})}_{L_t^2H^{s_1-1,s_2-1}}^2\\
&+\frac{3}{4}\normm{((\tau_{ij}^R)_{i\neq j})_{\Phi}}_{L_t^2H^{s_1-1,s_2-1}}^2+C\normm{(u_{\Phi}^R,\eps v_{\Phi}^R)}_{{L}^2_{t,\dot{\zeta}(t)}(H^{s_1-\f12,s_2-1})}^2,
\end{align*}	
where the constant depend on $s_1,s_2$, $\theta$, $\mathcal G_1$ and $\mathcal G_2$ but independent of $\eps$.
\end{mylem}

The proof of Lemmas \ref{time-tau}--\ref{K1+cdots+K16} will be presented in the end of this section. 

\subsection{Proof of Lemmas \ref{time-tau}--\ref{K1+cdots+K16}}
In this subsection, we present the proof of Lemmas \ref{time-tau}--\ref{K1+cdots+K16}.
\begin{proof}[Proof of Lemma \ref{time-tau}]
We first estimate $\p_tu$. Taking the $H^{s_1-1,s_2-1}(\Omega)$ inner product of \eqref{HOB-pyu} with $e^{2\frak{K}t}e^{a\langle D_x\rangle}\p_t\p_yu$, then there holds
\begin{align*}
&\normm{e^{\frak{K}t}e^{\frac{a}{2}\langle D_x\rangle}\p_t\p_yu}_{H^{s_1-1,s_2-1}}^2+\f12\frac{\mathrm{d}}{\mathrm{d}t}\normm{e^{\frak{K}t}e^{\frac{a}{2}\langle D_x\rangle}\p_y^2u}_{H^{s_1-1,s_2-1}}^2-\frak{K}\normm{e^{\frak{K}t}e^{\frac{a}{2}\langle D_x\rangle}\p_y^2u}_{H^{s_1-1,s_2-1}}^2\\
&+e^{2\frak{K}t}\inner{e^{\frac{a}{2}\langle D_x\rangle}(u\cdot\nabla_x\p_yu+v\p_y^2u), e^{\frac{a}{2}\langle D_x\rangle}\p_t\p_yu}_{H^{s_1-1,s_2-1}}\\
&+e^{2\frak{K}t}\inner{e^{\frac{a}{2}\langle D_x\rangle}(\p_yu\cdot\nabla_xu+\p_yv\p_yu), e^{\frac{a}{2}\langle D_x\rangle}\p_t\p_yu}_{H^{s_1-1,s_2-1}}\\
=&(1-\theta)e^{2\frak{K}t}\inner{e^{\frac{a}{2}\langle D_x\rangle}\p_y\mathcal F, e^{\frac{a}{2}\langle D_x\rangle}\p_t\p_yu}_{H^{s_1-1,s_2-1}},
\end{align*}
where we used the fact that $e^{\frac{a}{2}\langle D_x\rangle}(\p_t\p_yu)=\p_t(e^{\frac{a}{2}\langle D_x\rangle}\p_yu)$.
Hence integrating the resulting equality over $[0,t]$, we get
\beq\label{ptpyu}
\begin{split}
&\f12\normm{e^{\frak{K}t}e^{\frac{a}{2}\langle D_x\rangle}\p_t\p_yu}_{L_t^2H^{s_1-1,s_2-1}}^2+\f12\normm{e^{\frak{K}t}e^{\frac{a}{2}\langle D_x\rangle}\p_y^2u}_{L_t^\infty H^{s_1-1,s_2-1}}^2\\
\leqslant &\f12\normm{e^{a\langle D_x\rangle}\p_y^2u_{0}}_{H^{s_1-1,s_2-1}}^2+\frak{K}\normm{e^{\frak{K}t}e^{\frac{a}{2}\langle D_x\rangle}\p_y^2u}_{L_t^2H^{s_1-1,s_2-1}}^2\\
&+\f32\normm{e^{\frak{K}t}e^{\frac{a}{2}\langle D_x\rangle}(u\cdot\nabla_x\p_yu+v\p_y^2u)}_{L_t^2H^{s_1-1,s_2-1}}^2\\
&+\f32\normm{e^{\frak{K}t}e^{\frac{a}{2}\langle D_x\rangle}(\p_yu\cdot\nabla_xu+\p_yv\p_yu)}_{L_t^2H^{s_1-1,s_2-1}}^2\\
&+\f32(1-\theta)^2\normm{e^{\frak{K}t}e^{\frac{a}{2}\langle D_x\rangle}\p_y\mathcal F}_{L_t^2H^{s_1-1,s_2-1}}^2.
\end{split}
\eeq
Then we get, by applying Lemma \ref{lemma-product}, Lemma \ref{lemma-para-linearization-psi} ($s_1>\f52,s_2>\f32$), $\nabla_x\cdot u+\p_yv=0$ and Poincar\'e inequality,  
\begin{align*}
&\normm{e^{\frak{K}t}e^{\frac{a}{2}\langle D_x\rangle}(u\cdot\nabla_x\p_yu+v\p_y^2u)}_{L_t^2H^{s_1-1,s_2-1}}^2\\
\leqslant &C\normm{e^{\frac{a}{2}\langle D_x\rangle}u}_{{L}_t^\infty H^{s_1,s_2-1}}^2\left(\normm{e^{\frak{K}t}e^{\frac{a}{2}\langle D_x\rangle}\p_yu}_{{L}_t^2 H^{s_1,s_2-1}}^2+\normm{e^{\frak{K}t}e^{\frac{a}{2}\langle D_x\rangle}\p_y^2u}_{{L}_t^2 H^{s_1-1,s_2-1}}^2\right),
\end{align*}
and
\begin{align*}
\normm{e^{\frak{K}t}e^{\frac{a}{2}\langle D_x\rangle}(\p_yu\cdot\nabla_xu+\p_yv\p_yu)}_{L_t^2H^{s_1-1,s_2-1}}^2
\leqslant C\normm{e^{\frac{a}{2}\langle D_x\rangle}u}_{{L}_t^\infty H^{s_1,s_2-1}}^2\normm{e^{\frak{K}t}e^{\frac{a}{2}\langle D_x\rangle}\p_yu}_{{L}_t^2 H^{s_1,s_2-1}}^2.
\end{align*}
Similarly, recalling the definition of $\mathcal F$ in \eqref{FG1G2}, we have
\begin{align*}
&\normm{e^{\frak{K}t}e^{\frac{a}{2}\langle D_x\rangle}\p_y\mathcal F}_{L_t^2H^{s_1-1,s_2-1}}^2\\
\leqslant &C\left(\normm{e^{\frac{a}{2}\langle D_x\rangle}\p_yu}_{{L}_t^\infty H^{s_1-1,s_2}}^2+\normm{e^{\frac{a}{2}\langle D_x\rangle}\p_yu}_{{L}_t^\infty H^{s_1-1,s_2}}^6\right)\normm{e^{\frak{K}t}e^{\frac{a}{2}\langle D_x\rangle}\p_y^2u}_{{L}_t^2 H^{s_1-1,s_2}}^2,
\end{align*}
where the constant $C$ depends only on $s_1,s_2$, $\mathcal G_1$ and $\mathcal G_2$. Now with the above two estimates and then using \eqref{energy inequality}, we find that if the constant $c_1$ in \eqref{initial} is small sufficiently, then 
\beq\label{esti-ptpyu}
\begin{split}
&\normm{e^{\frak{K}t}e^{\frac{a}{2}\langle D_x\rangle}\p_t\p_yu}_{L_t^2H^{s_1-1,s_2-1}}^2+\normm{e^{\frak{K}t}e^{\frac{a}{2}\langle D_x\rangle}\p_y^2u}_{L_t^\infty H^{s_1-1,s_2-1}}^2\\
\leqslant &C\left(\normm{e^{a\langle D_x\rangle}u_0}_{H^{s_1,s_2}}^2+\normm{e^{a\langle D_x\rangle}\p_yu_0}_{H^{s_1,s_2}}^2\right).		
\end{split}
\eeq
Finally, the estimate for $\p_tu$ follows from Poincar\'e inequality \eqref{Poincare} and \eqref{esti-ptpyu}.

Then we estimate $\p_t\tau_{ij}$. Notice that (recalling the definition of $\mathcal G_1(\p_yu)$ in \eqref{FG1G2})
\[
\tau_{23}=(1-\theta)\frac{\p_yu_2}{1+\sigma\abs{\p_yu}^2}=(1-\theta)\left(\mathcal G_1(\p_yu)+1\right)\p_yu_2.
\]	
Hence by virtue of $s_1>\f52,s_2>\f32$, it follows from Lemma \ref{lemma-product}--\ref{lemma-para-linearization-psi} and \eqref{energy inequality} that 
\beq\label{esti-tau23}
\begin{split}
\normm{e^{\frac{a}{2}\langle D_x\rangle}\tau_{23}}_{L_t^\infty H^{s_1,s_2}}\leqslant &C\normm{e^{\frac{a}{2}\langle D_x\rangle}\p_yu}_{L_t^\infty H^{s_1,s_2}}^2+\normm{e^{\frac{a}{2}\langle D_x\rangle}\p_yu}_{L_t^\infty H^{s_1,s_2}}\\
\leqslant &C\left(\normm{e^{a\langle D_x\rangle}(u_0,\p_yu_0)}_{H^{s_1,s_2}}^2+\normm{e^{a\langle D_x\rangle}(u_0,\p_yu_0)}_{H^{s_1,s_2}}\right)\\
\leqslant &C\normm{e^{a\langle D_x\rangle}(u_0,\p_yu_0)}_{H^{s_1,s_2}},	
\end{split}
\eeq
where we also used \eqref{initial} and $c_1$ small sufficiently in the last inequality. Moreover, applying the operator $\p_t$ to the expression of $\tau_{23}$, we have
\begin{align*}
\p_t\tau_{23}=-(1-\theta)2\sigma\left(\mathcal G_2(\p_yu)+1\right)(\p_yu\cdot\p_t\p_yu)\p_yu_2+(1-\theta)\left(\mathcal G_1(\p_yu)+1\right)\p_t\p_yu_2.
\end{align*}
Thus we deduce from Lemmas \ref{lemma-product}--\ref{lemma-para-linearization-psi} and \eqref{esti-ptpyu} that
\begin{align*}
\normm{e^{\frac{a}{2}\langle D_x\rangle}\p_t\tau_{23}}_{L_t^2 H^{s_1-1,s_2-1}}\leqslant &C\left(1+\normm{e^{\frac{a}{2}\langle D_x\rangle}\p_yu}_{L_t^\infty H^{s_1-1,s_2-1}}\right)\left(1+\normm{e^{\frac{a}{2}\langle D_x\rangle}\p_yu}_{L_t^\infty H^{s_1-1,s_2-1}}^2\right)\\
&\times\normm{e^{\frac{a}{2}\langle D_x\rangle}\p_t\p_yu}_{L_t^2 H^{s_1-1,s_2-1}}\\
\leqslant &C\left(1+\normm{e^{a\langle D_x\rangle}(u_0,\p_yu_0)}_{H^{s_1,s_2}}\right)\left(1+\normm{e^{a\langle D_x\rangle}(u_0,\p_yu_0)}_{H^{s_1,s_2}}^2\right)\\
&\times\normm{e^{a\langle D_x\rangle}(u_0,\p_yu_0)}_{H^{s_1,s_2}}.
\end{align*}
Hence using \eqref{initial} and the smallness of $c_1$ in the last inequality, we have
\begin{align}\label{esti-pttau23}
\normm{e^{\frac{a}{2}\langle D_x\rangle}\p_t\tau_{23}}_{L_t^2 H^{s_1-1,s_2-1}}\leqslant C\normm{e^{a\langle D_x\rangle}(u_0,\p_yu_0)}_{H^{s_1,s_2}}.
\end{align}
Similarly, recalling 
\[
\tau_{13}=(1-\theta)\frac{\p_yu_1}{1+\sigma\abs{\p_yu}^2}=(1-\theta)\left(\mathcal G_1(\p_yu)+1\right)\p_yu_1.
\]
we have
\begin{align}
&\normm{e^{\frac{a}{2}\langle D_x\rangle}\tau_{13}}_{L_t^\infty H^{s_1,s_2}}\leqslant C\normm{e^{a\langle D_x\rangle}(u_0,\p_yu_0)}_{H^{s_1,s_2}},\label{esti-tau13}\\
&\normm{e^{\frac{a}{2}\langle D_x\rangle}\p_t\tau_{13}}_{L_t^2 H^{s_1-1,s_2-1}}\leqslant C\normm{e^{a\langle D_x\rangle}(u_0,\p_yu_0)}_{H^{s_1,s_2}}.\label{esti-pttau13}
\end{align}
Then by virtue of \eqref{Oldroyd-B-Prandtl2}$_1$--\eqref{Oldroyd-B-Prandtl2}$_4$, we obtain from Lemma \ref{lemma-product} and \eqref{energy inequality} that
\beq\label{esti-tau11223312}
\begin{split}
\normm{e^{\frac{a}{2}\langle D_x\rangle}(\tau_{11},\tau_{22},\tau_{33},\tau_{12})}_{L_t^\infty H^{s_1,s_2}}\leqslant &C\normm{e^{\frac{a}{2}\langle D_x\rangle}(\tau_{13},\tau_{23})}_{L_t^\infty H^{s_1,s_2}}\normm{e^{\frac{a}{2}\langle D_x\rangle}\p_yu}_{L_t^\infty H^{s_1,s_2}}\\
\leqslant &C\normm{e^{a\langle D_x\rangle}(u_0,\p_yu_0)}_{H^{s_1,s_2}},	
\end{split}
\eeq
where we also used the smallness of $c_1$ in the last inequality. Moreover, applying the estimate \eqref{esti-ptpyu}, \eqref{esti-pttau23} and \eqref{esti-pttau13}, it follows from Lemma \ref{lemma-product} that
\beq\label{esti-pttau11223312}
\begin{split}
&\normm{e^{\frac{a}{2}\langle D_x\rangle}(\p_t\tau_{11},\p_t\tau_{22},\p_t\tau_{33},\p_t\tau_{12})}_{L_t^2 H^{s_1-1,s_2-1}}\\
\leqslant &C\normm{e^{\frac{a}{2}\langle D_x\rangle}(\p_t\tau_{13},\p_t\tau_{23})}_{L_t^2 H^{s_1-1,s_2-1}}\normm{e^{\frac{a}{2}\langle D_x\rangle}\p_yu}_{L_t^\infty H^{s_1-1,s_2-1}}\\
&+C\normm{e^{\frac{a}{2}\langle D_x\rangle}(\tau_{13},\tau_{23})}_{L_t^\infty H^{s_1-1,s_2-1}}\normm{e^{\frac{a}{2}\langle D_x\rangle}\p_t\p_yu}_{L_t^2 H^{s_1-1,s_2-1}}\\
\leqslant &C\normm{e^{a\langle D_x\rangle}(u_0,\p_yu_0)}_{H^{s_1,s_2}},
\end{split}
\eeq
where we also used \eqref{energy inequality} and the smallness of $c_1$ in the last inequality.

The proof of Lemma \ref{time-tau} is thus complete.
\end{proof}

\begin{proof}[Proof of Lemma \ref{K1+cdots+K16}]
$\bullet$ \underline{Estimate on $K_{1}$.}\\
Recalling $s_1>\f52,s_2>\f32$ and
\[
\int_{\mathbb{T}}u^R(t,x,y)\mathrm{d}y=0,\quad \int_{\mathbb{T}}v^R(t,x,y)\mathrm{d}y=0.
\] 
Hence by virtue of $\p_xu+\p_yv=0$ and Poincar\'e inequality, it follows from Lemma \ref{lemma-product} that if the constant $c_1$ in \eqref{initial} small sufficiently, then
\begin{align*}
&-\int_0^t\inner{(u^R\cdot\nabla_xu)_\Phi, u_\Phi^R}_{H^{s_1-1,s_2-1}}+\eps^2\inner{(u^R\cdot\nabla_xv+v^R\p_yv)_\Phi, v_\Phi^R}_{H^{s_1-1,s_2-1}}\mathrm{d}t'\\
\leqslant &C\normm{(\nabla_xu_\Phi,\nabla_xv_\Phi)}_{L_t^\infty H^{s_1-1,s_2-1}}\normm{(u_\Phi^R,\eps v_\Phi^R)}_{L_t^2H^{s_1-1,s_2-1}}^2\\
\leqslant &C\normm{(u_\Phi,\nabla_xu_\Phi)}_{L_t^\infty H^{s_1,s_2}}\normm{\p_y(u_\Phi^R,\eps v_\Phi^R)}_{L_t^2H^{s_1-1,s_2-1}}^2\\
\leqslant &C\normm{e^{\frac{a}{2}\langle D_x\rangle}u}_{L_t^\infty H^{s_1,s_2}}\normm{\p_y(u_\Phi^R,\eps v_\Phi^R)}_{L_t^2H^{s_1-1,s_2-1}}^2\\
\leqslant &\frac{1}{32}\theta\normm{\p_y(u_\Phi^R,\eps v_\Phi^R)}_{L_t^2H^{s_1-1,s_2-1}}^2,
\end{align*}
where we used \eqref{upperlow} in the third inequality and \eqref{energy inequality} in the last inequality. On other hand, performing the similar argument for estimating $B_2$ (See \eqref{B2}) in the proof of Lemma \ref{nonlinear}, it holds that ($s_1>\f52$)
\beqo
\begin{split}
-\int_0^t\inner{(v^R\p_yu)_\Phi, u_\Phi^R}_{H^{s_1-1,s_2-1}}\mathrm{d}t'
\leqslant &C\int_0^t\normm{e^{\frac{a}{2}\langle D_x\rangle}\p_yu}_{H^{s_1,s_2}}\normm{u_\Phi^R}_{H^{s_1-\f12,s_2-1}}\normm{v_\Phi^R}_{H^{s_1-\f32,s_2-1}}\mathrm{d}t'\\
\leqslant &C\normm{(u_{\Phi}^R,\eps v_{\Phi}^R)}_{{L}^2_{t,\dot{\zeta}(t)}(H^{s_1-\f12,s_2-1})}^2.		
\end{split}
\eeqo
\noindent
Consequently, we have
\beq\label{K1}
\begin{split}
K_1\leqslant &C\normm{(u_{\Phi}^R,\eps v_{\Phi}^R)}_{{L}^2_{t,\dot{\zeta}(t)}(H^{s_1-\f12,s_2-1})}^2+\frac{1}{32}\theta\normm{\p_y(u_\Phi^R,\eps v_\Phi^R)}_{L_t^2H^{s_1-1,s_2-1}}^2.		
\end{split}
\eeq
\noindent
$\bullet$ \underline{Estimate on $K_{2}$.}\\ Along the same line as above, we have 
\begin{align*}
&-\int_0^t\inner{(u^\eps\cdot\nabla_xu^R)_\Phi, u_\Phi^R}_{H^{s_1-1,s_2-1}}+\eps^2\inner{(u^\eps\cdot\nabla_xv^R)_\Phi, v_\Phi^R}_{H^{s_1-1,s_2-1}}\mathrm{d}t'\\
\leqslant &C\int_0^t\normm{e^{\frac{a}{2}\langle D_x\rangle}\p_yu^\eps}_{H^{s_1,s_2}}\normm{(u_\Phi^R,\eps v_\Phi^R)}_{H^{s_1-\f12,s_2-1}}^2\mathrm{d}t'\\
\leqslant &C\normm{(u_{\Phi}^R,\eps v_{\Phi}^R)}_{{L}^2_{t,\dot{\zeta}(t)}(H^{s_1-\f12,s_2-1})}^2.
\end{align*}
Then applying the assumptions \eqref{assump-eps} in Theorem \ref{th1.2}, we find that if the constant $c_1$ in \eqref{assump-eps} small sufficiently, then the Poincar\'e inequality and $\nabla_x\cdot u^\eps+\p_yv^\eps=0$ implies that
\beqo
\begin{split}
&-\int_0^t\inner{(v^\eps\p_yu^R)_\Phi, u_\Phi^R}_{H^{s_1-1,s_2-1}}-\eps^2\inner{v^\eps\p_yv^R)_\Phi, v_\Phi^R}_{H^{s_1-1,s_2-1}}\mathrm{d}t'\\
\leqslant &C\normm{e^{\frac{a}{2}\langle D_x\rangle}u^\eps}_{L_t^\infty H^{s_1,s_2}}\normm{\p_y(u_\Phi^R,\eps v_\Phi^R)}_{L_t^2H^{s_1-1,s_2-1}}^2\\
\leqslant &\frac{1}{32}\theta\normm{\p_y(u_\Phi^R,\eps v_\Phi^R)}_{L_t^2H^{s_1-1,s_2-1}}^2.	
\end{split}
\eeqo
\noindent
Thus we have
\beq\label{K2}
\begin{split}
K_2\leqslant &C\normm{(u_{\Phi}^R,\eps v_{\Phi}^R)}_{{L}^2_{t,\dot{\zeta}(t)}(H^{s_1-\f12,s_2-1})}^2+\frac{1}{32}\theta\normm{\p_y(u_\Phi^R,\eps v_\Phi^R)}_{L_t^2H^{s_1-1,s_2-1}}^2.		
\end{split}
\eeq

\noindent
$\bullet$ \underline{Estimate on $K_{3}$ and $K_{4}$.}\\ Due to $s_1>\f52,s_2>\f32$, we obtain from Lemma \ref{lemma-product} that
\begin{align*}
K_3+K_4\leqslant &12\normm{((\tau_{ij})_\Phi)}_{L_t^\infty H^{s_1-1,s_2-1}}\normm{\p_yu^R_\Phi}_{L_t^2H^{s_1-1,s_2-1}}\normm{((\tau_{ij}^R)_\Phi)}_{L_t^2H^{s_1-1,s_2-1}}\\
&+12\normm{u_\Phi^\eps}_{L_t^\infty H^{s_1-1,s_2}}	\normm{((\tau_{ij}^R)_\Phi)}_{L_t^2H^{s_1-1,s_2-1}}^2.
\end{align*}
Hence we get, by virtue of \eqref{esti-tau11223312}, \eqref{esti-tau23}, \eqref{esti-tau13} and \eqref{initial}, that if the constant $c_1$ in \eqref{assump-eps} small sufficiently, then
\beq\label{K3+K4}
K_3+K_4\leqslant \frac{1}{32}\normm{((\tau_{ij}^R)_\Phi)}_{L_t^2H^{s_1-1,s_2-1}}^2+\frac{1}{32}\theta\normm{\p_yu^R_\Phi}_{L_t^2H^{s_1-1,s_2-1}}^2.
\eeq

\noindent
$\bullet$ \underline{Estimate on $K_{5}$ and $K_{6}$.}\\ Firstly, it follows from \eqref{upperlow} and Lemma \ref{time-tau} that
\begin{align*}
&-\f12\eps\sum_{i=1}^3\int_0^t\inner{(\p_t\tau_{ii})_\Phi, (\tau_{ii}^R)_\Phi}_{H^{s_1-1,s_2-1}}\mathrm{d}t'\\
\leqslant &\f12\eps\normm{e^{\frac{a}{2}\langle D_x\rangle}(\p_t\tau_{11},\p_t\tau_{22},\p_t\tau_{33})}_{L_t^2H^{s_1-1,s_2-1}}\normm{(\tau_{11}^R)_\Phi,(\tau_{22}^R)_\Phi,(\tau_{33}^R)_\Phi}_{L_t^2H^{s_1-1,s_2-1}}\\
\leqslant &C\eps^2+\frac{1}{128}\normm{(\tau_{11}^R)_\Phi,(\tau_{22}^R)_\Phi,(\tau_{33}^R)_\Phi}_{L_t^2H^{s_1-1,s_2-1}}^2.
\end{align*}	
Then applying Poincar\'e inequality, $\nabla_xu^\eps+\p_yv^\eps=0$ and Lemma \ref{lemma-product}, it yields that if the constant $c_1$ in \eqref{assump-eps} small sufficiently, then
\begin{align*}
&-\f12\eps\sum_{i=1}^3\int_0^t\inner{\left(u^\eps\cdot\nabla_x\tau_{ii}^\eps+v^\eps\p_y\tau_{ii}^\eps\right)_\Phi, (\tau_{ii}^R)_\Phi}_{H^{s_1-1,s_2-1}}\mathrm{d}t'\\
\leqslant &C\eps\normm{(u_{\Phi}^\eps,v_{\Phi}^\eps)}_{L_t^\infty H^{s_1-1,s_2-1}}\normm{(\nabla_x,\p_y)((\tau_{11}^\eps)_{\Phi},(\tau_{22}^\eps)_{\Phi},(\tau_{33}^\eps)_{\Phi})}_{L_t^2 H^{s_1-1,s_2-1}}\\
&\times\normm{(\tau_{11}^R)_\Phi,(\tau_{12}^R)_\Phi,(\tau_{22}^R)_\Phi}_{L_t^2H^{s_1-1,s_2-1}}\\
\leqslant &C\eps\normm{e^{\frac{a}{2}\langle D_x\rangle}u_{\Phi}^\eps}_{L_t^\infty H^{s_1,s_2}}\normm{e^{\frac{a}{2}\langle D_x\rangle}(\tau_{11}^\eps,\tau_{22}^\eps,\tau_{33}^\eps)}_{L_t^2 H^{s_1,s_2}}\normm{(\tau_{11}^R)_\Phi,(\tau_{22}^R)_\Phi,(\tau_{33}^R)_\Phi}_{L_t^2H^{s_1-1,s_2-1}}\\
\leqslant &C\eps^2+\frac{1}{128}\normm{(\tau_{11}^R)_\Phi,(\tau_{22}^R)_\Phi,(\tau_{33}^R)_\Phi}_{L_t^2H^{s_1-1,s_2-1}}^2.
\end{align*}
\noindent
Consequently, we have
\beq\label{K5}
K_5\leqslant C\eps^2+\frac{1}{64}\normm{(\tau_{11}^R)_\Phi,(\tau_{22}^R)_\Phi,(\tau_{33}^R)_\Phi}_{L_t^2H^{s_1-1,s_2-1}}^2.
\eeq

Notice that the terms in $K_5$ and $K_6$ have the same structure, hence repeating the above argument with slight modification, we also get
\beq\label{K6}
K_6\leqslant C\eps^2+\frac{1}{64}\normm{(\tau_{12}^R)_\Phi,(\tau_{13}^R)_\Phi,(\tau_{23}^R)_\Phi}_{L_t^2H^{s_1-1,s_2-1}}^2.
\eeq

\noindent
$\bullet$ \underline{Estimate on $K_{7}$--$K_{10}$.}\\ It follows from $\nabla_x\cdot u^\eps+\p_yv^\eps=0$ and Lemma \ref{lemma-product} that
\begin{align*}
K_7+\cdots+K_{10}\leqslant &C\eps\normm{u_{\Phi}^\eps}_{L_t^\infty H^{s_1,s_2-1}}\normm{((\tau_{ij}^\eps)_{\Phi})}_{L_t^2 H^{s_1-1,s_2-1}}\normm{((\tau_{ij}^R)_{\Phi})}_{L_t^2 H^{s_1-1,s_2-1}}\\
\leqslant &C\eps^2\normm{u_{\Phi}^\eps}_{L_t^\infty H^{s_1,s_2-1}}^2\normm{((\tau_{ij}^\eps)_{\Phi})}_{L_t^2 H^{s_1-1,s_2-1}}^2+\frac{1}{64}\normm{((\tau_{ij}^R)_\Phi)}_{L_t^2H^{s_1-1,s_2-1}}^2.
\end{align*}
In particular, if the smallness of $c_1$ in \eqref{assump-eps} small sufficiently, then it holds that
\beq\label{K7+K8+K9+K10}
K_7+\cdots+K_{10}\leqslant C\eps^2+\frac{1}{64}\normm{((\tau_{ij}^R)_\Phi)}_{L_t^2H^{s_1-1,s_2-1}}^2.
\eeq

\noindent
$\bullet$ \underline{Estimate on $K_{11}$--$K_{12}$.}\\ We get, by applying Lemma \ref{lemma-product}, that
\beq\label{K11+K12}
\begin{split}
K_{11}+K_{12}\leqslant &C\eps\normm{\eps v_{\Phi}^\eps}_{L_t^\infty H^{s_1,s_2-1}}\normm{((\tau_{ij}^\eps)_{\Phi})}_{L_t^2 H^{s_1-1,s_2-1}}\normm{((\tau_{ij}^R)_{\Phi})}_{L_t^2 H^{s_1-1,s_2-1}}\\
\leqslant &C\eps^2+\frac{1}{64}\normm{((\tau_{ij}^R)_\Phi)}_{L_t^2H^{s_1-1,s_2-1}}^2,
\end{split}
\eeq
where we also used \eqref{assump-eps} and the smallness of $c_1$.

\noindent
$\bullet$ \underline{Estimate on $K_{13}$.}\\ We deduce from Lemma \ref{lemma-product}, Poincar\'e inequality, \eqref{psilow} and \eqref{upperlow} that
\begin{align*}
&-\eps^2\int_0^t\inner{\left(u\cdot\nabla_xv+v\p_yv\right)_\Phi, v^R_\Phi}_{H^{s_1-1,s_2-1}}\mathrm{d}t'\\
\leqslant &\eps\normm{(u_{\Phi},\p_yv_{\Phi})}_{L_t^2H^{s_1-1,s_2-1}}\normm{(\nabla_xv_{\Phi},v_{\Phi})}_{L_t^\infty H^{s_1-1,s_2-1}}\normm{\eps v_{\Phi}^R}_{L^2H^{s_1-1,s_2-1}}\\
\leqslant &\eps\normm{e^{\frac{a}{2}\langle D_x\rangle}\p_yu}_{L_t^2H^{s_1,s_2}}\normm{e^{\frac{a}{2}\langle D_x\rangle}u}_{L_t^\infty H^{s_1,s_2}}\normm{\eps\p_y v_{\Phi}^R}_{L^2H^{s_1-1,s_2-1}}\\
\leqslant &C\eps^2+\frac{1}{64}\theta	\normm{\eps\nabla_xu_{\Phi}^R}_{L^2H^{s_1-1,s_2-1}}^2,
\end{align*}
where we used \eqref{energy inequality} in the last inequality. Then by virtue of Lemma \ref{time-tau}, it follows from Poincar\'e inequality, \eqref{psilow} and \eqref{upperlow} that
\begin{align*}
\normm{\p_tv_\Phi}_{L_t^2H^{s_1-1,s_2-1}}\leqslant &	2\normm{\p_t\nabla_x\cdot u_\Phi}_{L_t^2H^{s_1-1,s_2-1}}\\
\leqslant &2\normm{e^{\frac{a}{2}\langle D_x\rangle}\p_tu}_{L_t^2H^{s_1-1,s_2-1}}\leqslant C\normm{e^{a\langle D_x\rangle}(u_0,\p_yu_0)}_{H^{s_1,s_2}}.
\end{align*}
Thus we obtain from integration by parts, $\nabla_x\cdot u^\eps+\p_yv^\eps=0$ and \eqref{upperlow} that
\begin{align*}
&\eps^2\int_0^t\inner{\theta\Delta_xu_\Phi, u^R_\Phi}_{H^{s_1-1,s_2-1}}-\inner{\left(\p_tv-\theta\Delta_{\eps}v\right)_\Phi, v^R_\Phi}_{H^{s_1-1,s_2-1}}\mathrm{d}t'\\
\leqslant &C\eps\left(\normm{(\nabla_xu_{\Phi},\nabla_xv_{\Phi})}_{L_t^2H^{s_1-1,s_2-1}}+\normm{\p_tv_{\Phi}}_{L_t^2H^{s_1-1,s_2-1}}\right)\normm{(\eps\nabla_x,\p_y)(u_{\Phi}^R,\eps v_{\Phi}^R)}_{L^2H^{s_1-1,s_2-1}}\\
\leqslant &C\eps\left(\normm{u_{\Phi}}_{L_t^2H^{s_1+1,s_2-1}}+\normm{\p_tv_{\Phi}}_{L_t^2H^{s_1-1,s_2-1}}\right)\normm{(\eps\nabla_x,\p_y)(u_{\Phi}^R,\eps v_{\Phi}^R)}_{L^2H^{s_1-1,s_2-1}}\\
\leqslant &C\eps\left(\normm{\p_yu_{\Phi}}_{L_t^2H^{s_1+1,s_2-1}}+\normm{\p_tv_{\Phi}}_{L_t^2H^{s_1-1,s_2-1}}\right)\normm{\eps\nabla_x(u_{\Phi}^R,\eps v_{\Phi}^R)}_{L^2H^{s_1-1,s_2-1}}\\
\leqslant &C\eps\left(\normm{e^{\frac{a}{2}\langle D_x\rangle}\p_yu}_{L_t^2H^{s_1,s_2}}+\normm{e^{\frac{a}{2}\langle D_x\rangle}\p_tu}_{L_t^2H^{s_1-1,s_2-1}}\right)\normm{\eps\nabla_x(u_{\Phi}^R,\eps v_{\Phi}^R)}_{L^2H^{s_1-1,s_2-1}}\\
\leqslant &C\eps^2+\frac{1}{64}\theta	\normm{\eps\nabla_x(u_{\Phi}^R,\eps v_{\Phi}^R)}_{L^2H^{s_1-1,s_2-1}}^2,
\end{align*}
where we used Poincar\'e inequality in the third inequality and \eqref{energy inequality} in the last inequality. 

In summary, we have
\beq\label{K13}
K_{13}\leqslant C\eps^2+\frac{1}{32}\theta	\normm{\eps\nabla_x(u_{\Phi}^R,\eps v_{\Phi}^R)}_{L^2H^{s_1-1,s_2-1}}^2.
\eeq

\noindent
$\bullet$ \underline{Estimate on $K_{14}$.}\\ Recalling the definition of $K_{14}$, we get, by applying the integration by parts, that
\beq\label{K14}
\begin{split}
K_{14}=&\eps\int_0^t\inner{(\p_{x_1}\tau_{11}^\eps+\p_{x_2}\tau_{12}^\eps-\p_{x_1}\tau_{33}^\eps)_\Phi, (u_1^R)_\Phi}_{H^{s_1-1,s_2-1}}\mathrm{d}t'\\
&+\eps\int_0^t\inner{(\p_{x_1}\tau_{12}^\eps+\p_{x_2}\tau_{22}^\eps-\p_{x_2}\tau_{33}^\eps)_\Phi, (u_2^R)_\Phi}_{H^{s_1-1,s_2-1}}\mathrm{d}t'\\
&+\eps\int_0^t\inner{(\p_{x_1}\tau_{31}^\eps+\p_{x_2}\tau_{32}^\eps)_\Phi, \eps v^R_\Phi}_{H^{s_1-1,s_2-1}}\mathrm{d}t'\\
\leqslant &C\eps\normm{(\tau_{ij}^\eps)_\Phi}_{L_t^2H^{s_1,s_2}}\normm{(u_\Phi^R,\eps v_{\Phi}^R)}_{L_t^2H^{s_1-1,s_2-1}}\\
\leqslant &C\eps^2\normm{e^{\frac{a}{2}\langle D_x\rangle}(\tau_{ij}^\eps)}_{L_t^2H^{s_1,s_2}}^2+\frac{1}{32}\theta\normm{(\eps\nabla_x,\p_y)u_\Phi^R}_{L_t^2H^{s_1-1,s_2-1}}^2\\
\leqslant &C\eps^2+\frac{1}{32}\theta\normm{(\eps\nabla_x,\p_y)u_\Phi^R}_{L_t^2H^{s_1-1,s_2-1}}^2,	
\end{split}
\eeq
where we used \eqref{upperlow} in the second inequality and the assumptions \eqref{assump-eps} in the last inequality.

\noindent
$\bullet$ \underline{Estimate on $K_{15}$.}\\ Due to the fact that $\nabla_x\cdot u^\eps+\p_yv^\eps=0$ and 
\[
\int_{\mathbb{T}}u^\eps(t,x,y)\mathrm{d}y=0,\quad \int_{\mathbb{T}}v^\eps(t,x,y)\mathrm{d}y=0,
\]
we obtain from Poincar\'e inequality that
\beqo
\begin{split}
K_{15}\leqslant &C\eps^2\normm{(u^\eps_{\Phi},\eps v^\eps_{\Phi})}_{L_t^2H^{s_1,s_2}}^2+\frac{1}{32}	\normm{(\tau_{11}^R)_\Phi,(\tau_{12}^R)_\Phi,(\tau_{22}^R)_\Phi}_{L_t^2H^{s_1-1,s_2-1}}^2\\
\leqslant &C\eps^2\normm{\p_y(u^\eps_{\Phi},\eps v^\eps_{\Phi})}_{L_t^2H^{s_1,s_2}}^2+\frac{1}{32}	\normm{(\tau_{11}^R)_\Phi,(\tau_{12}^R)_\Phi,(\tau_{22}^R)_\Phi}_{L_t^2H^{s_1-1,s_2-1}}^2.
\end{split}
\eeqo
Then applying \eqref{upperlow} and the assumptions \eqref{assump-eps} in Theorem \ref{th1.2}, we have that 
\beq\label{K15}
\begin{split}
K_{15}\leqslant &C\eps^2\normm{e^{\frac{a}{2}\langle D_x\rangle}(\eps\nabla_x,\p_y)u^\eps}_{L_t^2H^{s_1,s_2}}^2+\frac{1}{32}	\normm{(\tau_{11}^R)_\Phi,(\tau_{12}^R)_\Phi,(\tau_{22}^R)_\Phi}_{L_t^2H^{s_1-1,s_2-1}}^2\\
\leqslant &C\eps^2+\frac{1}{32}	\normm{(\tau_{11}^R)_\Phi,(\tau_{12}^R)_\Phi,(\tau_{22}^R)_\Phi}_{L_t^2H^{s_1-1,s_2-1}}^2.
\end{split}
\eeq

\noindent
$\bullet$ \underline{Estimate on $K_{16}$.}\\ By integration by parts, it holds that
\beq\label{K16}
\begin{split}
K_{16}=&-\theta\int_0^t\inner{\p_y(u_1^R)_\Phi, (\tau_{13}^R)_\Phi}_{H^{s_1-1,s_2-1}}+\inner{\p_y(u_2^R)_\Phi, (\tau_{23}^R)_\Phi}_{H^{s_1-1,s_2-1}}\mathrm{d}t'\\
\leqslant&\f12\theta\normm{\p_yu^R_\Phi}_{L_t^2H^{s_1-1,s_2-1}}^2+\f12\normm{((\tau_{13}^R)_\Phi,(\tau_{23}^R)_\Phi)}_{L_t^2H^{s_1-1,s_2-1}}^2.	
\end{split}
\eeq

Now Lemma \ref{K1+cdots+K16} follows from \eqref{K1}, \eqref{K2}, \eqref{K3+K4}, \eqref{K5}, \eqref{K6}, \eqref{K7+K8+K9+K10}, \eqref{K11+K12}, \eqref{K13}, \eqref{K14}, \eqref{K15}, \eqref{K16}.
\end{proof}

\appendix
\section{Derivation of \eqref{tau13+tau23}}\label{Appendix A}
In this appendix, we present the detailed derivation of \eqref{tau13+tau23}. Indeed, it follows from \eqref{Oldroyd-B-Prandtl2}$_1$--\eqref{Oldroyd-B-Prandtl2}$_3$ that
\begin{align*}
&\tau_{11}+\tau_{33}=-2b\tau_{13}\p_yu_1-(b+1)\tau_{23}\p_yu_2,\\	
&\tau_{11}-\tau_{33}=2\tau_{13}\p_yu_1+(b+1)\tau_{23}\p_yu_2,\\
&\tau_{22}+\tau_{33}=-2b\tau_{23}\p_yu_2-(b+1)\tau_{13}\p_yu_1,\\	
&\tau_{22}-\tau_{33}=2\tau_{23}\p_yu_2+(b+1)\tau_{13}\p_yu_1.
\end{align*}
Hence, we further use \eqref{Oldroyd-B-Prandtl2}$_5$--\eqref{Oldroyd-B-Prandtl2}$_6$, to compute 
\begin{align*}
A\tau_{13}+B\tau_{23}=2(1-\theta)\p_yu_1,\\
B\tau_{13}+C\tau_{23}=2(1-\theta)\p_yu_2,	
\end{align*}
where
\begin{align*}
&A=2+2(1-b^2)(\p_yu_1)^2-\f12(b^2-1)(\p_yu_2)^2,\\
&B=\f32(1-b^2)\p_yu_2\p_yu_1,\\
&C=2+2(1-b^2)(\p_yu_2)^2-\f12(b^2-1)(\p_yu_1)^2	.
\end{align*}
Thus
\[
\tau_{23}=2(1-\theta)\frac{B\p_yu_1-A\p_yu_2}{B^2-AC}.
\]
Notice that
\begin{align*}
&B^2-AC=-\left[1+(1-b^2)((\p_yu_1)^2+(\p_yu_2)^2)\right]\big[4+(1-b^2)((\p_yu_1)^2+(\p_yu_2)^2)\big],\\
&B\p_yu_1-A\p_yu_2=-\f12\left[4+(1-b^2)((\p_yu_1)^2+(\p_yu_2)^2)\right]\p_yu_2,
\end{align*}
we find 
\[
\tau_{23}=\frac{(1-\theta)\p_yu_2}{1+(1-b^2)((\p_yu_1)^2+(\p_yu_2)^2)}.
\]
Similarly, we have
\[
\tau_{13}=2(1-\theta)\frac{B\p_yu_2-C\p_yu_1}{B^2-AC}=\frac{(1-\theta)\p_yu_1}{1+(1-b^2)((\p_yu_1)^2+(\p_yu_2)^2)}.
\]

\section{Proof of some technical lemmas}\label{Appendix B}
In this appendix, we present some technical lemmas which is used throughout our paper.
\begin{mylem}\label{lem:f+}
	For any $a,b\in L^2(\R^2\times\mathbb{T})$, define
	\beq\label{hat-a+}
		a^+ = \mathcal{F}^{-1}(|\mathcal{F}(a)|) \quad \mbox{and} \quad b^+ = \mathcal{F}^{-1}(|\mathcal{F}(b)|).
	\eeq
	Then the following inequality holds: 
	\begin{displaymath}
		\abs{(\widehat{ab})_\Psi(\xi,k)} \leqslant \widehat{a^+_\Psi a^+_\Psi }(\xi,k).
	\end{displaymath}
\end{mylem} 
\begin{proof}
Noting that the phase function is convex, therefore  
\begin{align*}
\abs{(\widehat{ab})_\Psi(\xi,k)} &= e^{\Psi(\xi,k)} \abs{\widehat{a}(\cdot)*\widehat{b}(\cdot)(\xi,k)}\\
&\leqslant \sum_{m=-\infty}^\infty\int_{\R^2} e^{\Psi(\xi-\eta,k-m)}\abs{\widehat{a}(\xi - \eta,k-m)}e^{\Psi(\eta,m)}\abs{\widehat{b}(\eta,m)}\mathrm{d}\eta \\
&\leqslant \abs{\widehat{a}_\Psi}*\abs{\widehat{b}_\Psi}(\xi,k) = \widehat{a}^+_\Psi * \widehat{b}^+_\Psi = \widehat{a^+_\Psi b^+_\Psi }(\xi,k)
\end{align*}
which complete the proof of this lemma.
\end{proof}

\begin{mylem}\label{lemma-product}
Let $f,g$ be two functions such that $f_\Psi$, $g_\Psi\in H^{s_1,s_2}(\R^2\times\mathbb{T})$ for some $s_1>1,s_2>\f12$, where the weight function $\Psi$ is defined as $\Psi=a(1+|\xi|)$ with $a>0$. Then there holds
\begin{equation*}
\normm{(f g)_\Psi}_{H^{s_1,s_2}} 
        \leqslant C_{s_1,s_2}\normm{f_\Psi}_{H^{s_1,s_2}}\normm{g_\Psi}_{H^{s_1,s_2}}
\end{equation*}
for some positive constant $C_{s_1,s_2}$ which depend on $s_1,s_2$.
\end{mylem}
\begin{proof}
For simplicity, we use the classical Japanese bracket $\langle\xi\rangle:=(1+\abs{\xi}^2)^{\f12}$. The \eqref{hat-a+} implies that $\normm{a^+_\Psi}_{H^{s_1,s_2}}=\normm{a_\Psi}_{H^{s_1,s_2}}$. Hence we deduce from Lemma \ref{lem:f+} that
\begin{align*}
\normm{(fg)_{\Psi}}_{H^{s_1,s_2}}^2=&\sum_{k\in\Z}\int_{\mathbb{R}}\langle\xi\rangle^{2s_1}\langle k\rangle^{2s_2}\abs{(\widehat{fg})_\Psi(\xi,k)}^2\mathrm{d}\xi\\
\leqslant&\sum_{k\in\Z}\int_{\mathbb{R}}\langle\xi\rangle^{2s_1}\langle k\rangle^{2s_2}\abs{\widehat{f^+_\Psi g^+_\Psi }(\xi,k)}^2\mathrm{d}\xi\\
=&\normm{\langle\xi\rangle^{s_1}\langle k\rangle^{s_2}\sum_{m\in\Z}\int_{\R^2}\widehat{f^+_\Psi}(\xi-\eta,k-m)\widehat{g^+_\Psi}(\eta,m)\mathrm{d}\eta}_{L_{\xi}^2\ell_k^2}^2.
\end{align*}
Notice that the classical inequality $\abs{\alpha\pm\beta}^s\leqslant\max(1,2^{s-1})(\abs{\alpha}^s+\abs{\beta}^s)$ implies that
\[
\langle\xi\rangle^{s_1}\langle k\rangle^{s_2}\leqslant C_{s_1,s_2}\left(\langle\xi-\eta\rangle^{s_1}+\langle\eta\rangle^{s_1}\right)\left(\langle k-m\rangle^{s_2}+\langle m\rangle^{s_2}\right).
\]
Therefore we obtain from the Young's inequality that
\begin{align*}
\normm{(fg)_{\Psi}}_{H^{s_1,s_2}}
\leqslant &C_{s_1,s_2}\normm{\langle\xi\rangle^{s_1}\langle k\rangle^{s_2}\widehat{f^+_\Psi}(\xi,k)}_{L_{\xi}^2\ell_k^2}\normm{\widehat{g^+_\Psi}(\xi,k)}_{L_{\xi}^1\ell_k^1}\\
&+C_{s_1,s_2}\normm{\langle\xi\rangle^{s_1}\widehat{f^+_\Psi}(\xi,k)}_{L_{\xi}^2\ell_k^1}\normm{\langle k\rangle^{s_2}\widehat{g^+_\Psi}(\xi,k)}_{L_{\xi}^1\ell_k^2}\\
&+C_{s_1,s_2}\normm{\langle k\rangle^{s_2}\widehat{f^+_\Psi}(\xi,k)}_{L_{\xi}^1\ell_k^2}\normm{\langle\xi\rangle^{s_1}\widehat{g^+_\Psi}(\xi,k)}_{L_{\xi}^2\ell_k^1}\\
&+C_{s_1,s_2}\normm{\widehat{f^+_\Psi}(\xi,k)}_{L_{\xi}^1\ell_k^1}\normm{\langle\xi\rangle^{s_1}\langle k\rangle^{s_2}\widehat{g^+_\Psi}(\xi,k)}_{L_{\xi}^2\ell_k^2}\\
\leqslant &C_{s_1,s_2}\normm{\langle\xi\rangle^{s_1}\langle k\rangle^{s_2}\widehat{f^+_\Psi}(\xi,k)}_{L_{\xi}^2\ell_k^2}\normm{\langle\xi\rangle^{s_1}\langle k\rangle^{s_2}\widehat{g^+_\Psi}(\xi,k)}_{L_{\xi}^2\ell_k^2}\\
=&C_{s_1,s_2}\normm{f_\Psi}_{H^{s_1,s_2}}\normm{g_\Psi}_{H^{s_1,s_2}},
\end{align*}
where we also used the fact that $\langle\xi\rangle^{-s_1}\in L^2(\R^2)$ and $\langle k\rangle^{-s_2}\in \ell^2$ for $s_1>1$ and $s_2>\f12$, respectively. The proof is thus complete.
\end{proof}

\begin{mylem}\label{lemma-para-linearization-psi}
Suppose that $s_1>1, s_2>\f12$ and $M_0>0$. Let $f$ be a holomorphic function in the ball $\{z \in \mathbb{C}\mid\abs{z}< M_0\}$ satisfying $f(0)=0$. The weight function $\Psi$ is defined as $\Psi=a(1+\abs{\xi})$ with $a>0$. Then there exists $\eps_0>0$ such that if $b_{\Psi} \in H^{s_1,s_2}(\Omega)$ which satisfies $\normm{b_{\Psi}}_{H^{s_1,s_2}} \leqslant \eps_0$, then $(f(b))_{\Psi}\in H^{s_1,s_2}(\Omega)$. More precisely, there exists $C>0$ depending only on $f, s_1, s_2$ such that
\[
\normm{(f(b))_{\Psi}}_{H^{s_1,s_2}} \leqslant C \normm{b_{\Psi}}_{H^{s_1,s_2}}.
\]
\end{mylem}
\begin{proof} It follows from an induction that for all $n \geqslant 1$,
\begin{equation}\label{u^n}
  \normm{(b^n)_{\Psi}}_{H^{s_1,s_2}} \leqslant (2C_{s_1,s_2})^{n-1} \normm{b_{\Psi}}^{n-1}_{H^{s_1,s_2}}\normm{b_{\Psi}}_{H^{s_1,s_2}}.
  \end{equation}
For $\abs{z}<M_0$ we can write $f(z)=\sum\limits_{n=1}^{+ \infty} a_n z^n$ where $a_n$ is such that $\abs{a_n}\leqslant K^n$ for some $K>0$. We will show that the series $\sum a_n b^n$ is uniformly convergent in $H^{s_1,s_2}(\Omega)$. Indeed, we deduce from \eqref{u^n} and the hypothesis that
\[
\normm{a_n(b^n)_{\Psi}}_{H^{s_1,s_2}} \leqslant K(2C_{s_1,s_2}K\eps_0)^{n-1}\normm{b_{\Psi}}_{H^{s_1,s_2}}.
\]
Now taking $\eps_0$ small enough such that $2C_{s_1,s_2}K \eps_0<\f12$. Since $\normm{b_{\Psi}}_{L^\infty(\R^2\times\mathbb{T})} \leqslant C_{s_1,s_2}\normm{b_{\Psi}}_{H^{s_1,s_2}}$, we also can set $\eps_0$ small enough such that $C_{s_1,s_2}\eps_0<M_0$. Then there holds 
\begin{align*}
\normm{(f(b))_{\Psi}}_{H^{s_1,s_2}} \leqslant &K\left(\sum_{n=1}^{+ \infty}(2C_{s_1,s_2}K)^{n-1}\normm{b_{\Psi}}^{n-1}_{H^{s_1,s_2}} \right)\normm{b_{\Psi}}_{H^{s_1,s_2}}\\
\leqslant &K\left(\sum_{n=1}^{+ \infty}(2C_{s_1,s_2}K\eps_0)^{n-1}\right)\normm{b_{\Psi}}_{H^{s_1,s_2}}\leqslant 2K\normm{b_{\Psi}}_{H^{s_1,s_2}}.
\end{align*}
This completes the proof of this lemma.
\end{proof}

\noindent{\bf Acknowledgements}
N. Zhu was partially supported by the National Natural Science Foundation of China (Grant No.12301285, No.12171010), and project ZR2023QA002 supported by the Shandong Provincial Natural Science Foundation.
\vskip .3in

\noindent {\bf Conflict of interest}. 
The authors declare that they have no conflict of interest.

\noindent {\bf Data availability}.
No data was used for the research described in the article.

\bibliographystyle{amsplain}

\begin{thebibliography}{10}
\bibitem{BL92}
O.~Besson and M.~R. Laydi, \emph{Some estimates for the anisotropic
	{N}avier-{S}tokes equations and for the hydrostatic approximation}, RAIRO
Mod\'{e}l. Math. Anal. Num\'{e}r. \textbf{26} (1992), no.~7, 855--865.

\bibitem{BreschPrange2014}
D.~Bresch and C.~Prange, \emph{Newtonian {L}imit for {W}eakly {V}iscoelastic
	{F}luid {F}lows}, SIAM J. Math. Anal. \textbf{46} (2014), no.~2, 1116--1159.

\bibitem{cailei2019}
Y.~Cai, Z.~Lei, F.~Lin, and N.~Masmoudi, \emph{Vanishing {V}iscosity {L}imit
	for {I}ncompressible {V}iscoelasticity in {T}wo {D}imensions}, Comm. Pure
Appl. Math. \textbf{72} (2019), no.~10, 2063--2120.

\bibitem{CheminMasmoudi2001}
J.-Y Chemin and N.~Masmoudi, \emph{About {L}ifespan of {R}egular {S}olutions of
	{E}quations {R}elated to {V}iscoelastic {F}luids}, SIAM J. Math. Anal.
\textbf{33} (2001), no.~1, 84--112.

\bibitem{ChenMiao2008}
Q.~Chen and C.~Miao, \emph{Global well-posedness of viscoelastic fluids of
	{O}ldroyd type in {B}esov spaces}, Nonlinear Anal. \textbf{68} (2008), no.~7,
1928--1939.

\bibitem{ConstantinKliegl2012}
P.~Constantin and M.~Kliegl, \emph{Note on global regularity for
	two-dimensional {O}ldroyd-{B} fluids with diffusive stress}, Arch. Ration.
Mech. Anal. \textbf{206} (2012), no.~3, 725--740.

\bibitem{DePaicu2020}
F.~De~Anna and M.~Paicu, \emph{The {F}ujita-{K}ato theorem for some
	{O}ldroyd-{B} model}, J. Funct. Anal. \textbf{279} (2020), no.~11, 108761,
64.

\bibitem{ElgindiLiu2015}
T.-M. Elgindi and J.~Liu, \emph{Global wellposedness to the generalized
	{O}ldroyd type models in {$\R^3$}}, J. Differential Equations \textbf{259}
(2015), no.~5, 1958--1966.

\bibitem{EligndiRousset2015}
T.-M. Elgindi and F.~Rousset, \emph{Global regularity for some {O}ldroyd-{B}
	type models}, Comm. Pure Appl. Math. \textbf{68} (2015), no.~11, 2005--2021.

\bibitem{FangHieberZi2013}
D.~Fang, M.~Hieber, and R.~Zi, \emph{Global existence results for {O}ldroyd-{B}
	fluids in exterior domains: the case of non-small coupling parameters}, Math.
Ann. \textbf{357} (2013), no.~2, 687--709. 

\bibitem{Fangzi2016}
D.~Fang and R.~Zi, \emph{Global solutions to the {O}ldroyd-{B} model with a
	class of large initial data}, SIAM J. Math. Anal. \textbf{48} (2016), no.~2,
1054--1084.

\bibitem{FengWangWu}
W~Feng, W.~Wang, and J.~Wu, \emph{Global existence and stability for the 2{D}
	{O}ldroyd-{B} model with mixed partial dissipation}, Proc. Amer. Math. Soc.
\textbf{150} (2022), no.~12, 5321--5334.

\bibitem{GuillopeSaut1990}
C.~Guillop\'{e} and J.-C. Saut, \emph{Existence results for the flow of
	viscoelastic fluids with a differential constitutive law}, Nonlinear Anal.
\textbf{15} (1990), no.~9, 849--869.

\bibitem{HieberNaitoShibata2012}
M.~Hieber, Y.~Naito, and Y.~Shibata, \emph{Global existence results for
	{O}ldroyd-{B} fluids in exterior domains}, J. Differential Equations
\textbf{252} (2012), no.~3, 2617--2629.

\bibitem{kessenich2009}
P.~Kessenich, \emph{Global existence with small initial data for
	three-dimensional incompressible isotropic viscoelastic materials}, ProQuest
LLC, Ann Arbor, MI, 2008, Thesis (Ph.D.)--University of California, Santa
Barbara.

\bibitem{LeiLiuZhou2008}
Z.~Lei, C.~Liu, and Y.~Zhou, \emph{Global solutions for incompressible
	viscoelastic fluids}, Arch. Ration. Mech. Anal. \textbf{188} (2008), no.~3,
371--398.

\bibitem{Lions96}
P.-L. Lions, \emph{{M}athematical topics in fluid mechanics. {V}ol. 1}, Oxford
Lecture Series in Mathematics and its Applications, vol.~3, The Clarendon
Press, Oxford University Press, New York, 1996, Incompressible models, Oxford
Science Publications.

\bibitem{LionsMasmoudi2000}
P.-L. Lions and N.~Masmoudi, \emph{Global solutions for some {O}ldroyd models
	of non-{N}ewtonian flows}, Chinese Ann. Math. Ser. B \textbf{21} (2000),
no.~2, 131--146.

\bibitem{MolinetTalhouk2004}
L.~Molinet and R.~Talhouk, \emph{On the global and periodic regular flows of
	viscoelastic fluids with a differential constitutive law}, NoDEA Nonlinear
Differential Equations Appl. \textbf{11} (2004), no.~3, 349--359.

\bibitem{Oldroyd1950}
J.~G. Oldroyd, \emph{On the formulation of rheological equations of state},
Proc. Roy. Soc. London Ser. A \textbf{200} (1950), 523--541.

\bibitem{Tao}
T.~Tao, \emph{{N}onlinear dispersive equations}, CBMS Regional Conference
Series in Mathematics, vol. 106, Published for the Conference Board of the
Mathematical Sciences, Washington, DC; by the American Mathematical Society,
Providence, RI, 2006, Local and global analysis.

\bibitem{WWX}
P.~Wang, J.~Wu, X.~Xu, and Y.~Zhong, \emph{Sharp decay estimates for
	{O}ldroyd-{B} model with only fractional stress tensor diffusion}, J. Funct.
Anal. \textbf{282} (2022), no.~4, Paper No. 109332, 55.

\bibitem{ZW}
J.~Zhao and J.~Wu, \emph{Oldroyd-{B} model with high {W}eissenberg number and
	fractional velocity dissipation}, J. Geom. Anal. \textbf{33} (2023), no.~9,
Paper No. 296, 38.

\bibitem{zhu2018}
Y.~Zhu, \emph{Global small solutions of 3{D} incompressible {O}ldroyd-{B} model
	without damping mechanism}, J. Funct. Anal. \textbf{274} (2018), no.~7,
2039--2060.

\bibitem{Zi2021}
R.~Zi, \emph{Vanishing viscosity limit of the 3{D} incompressible {O}ldroyd-{B}
	model}, Ann. Inst. H. Poincar\'{e} C Anal. Non Lin\'{e}aire \textbf{38}
(2021), no.~6, 1841--1867.

\bibitem{ZiFangZhang2014}
R.~Zi, D.~Fang, and T.~Zhang, \emph{Global solution to the incompressible
	{O}ldroyd-{B} model in the critical {$L^p$} framework: the case of the
	non-small coupling parameter}, Arch. Ration. Mech. Anal. \textbf{213} (2014),
no.~2, 651--687.	
\end{thebibliography}

\end{document}